\documentclass{mcom-l}
\usepackage{fullpage,relsize,cite,graphicx,mathrsfs,comment,caption,subcaption,mathtools}

\usepackage{graphicx}
\usepackage{amsfonts}
\usepackage{amsmath}
\usepackage{amsthm}
\usepackage{amssymb}
\usepackage{hyperref}
\usepackage{mathtools}
\usepackage{verbatim}

\newtheorem{theorem}{Theorem}[section]
\newtheorem{corollary}{Corollary}
\newtheorem{lemma}{Lemma}
\newtheorem{proposition}{Proposition}
\newtheorem{definition}{Definition}
\newtheorem{remark}{Remark}
\newtheorem{problem}{Problem}
\newtheorem{assumption}{Assumption}

\numberwithin{equation}{section}
\numberwithin{lemma}{section}
\numberwithin{corollary}{section}
\numberwithin{theorem}{section}
\numberwithin{proposition}{section}
\numberwithin{problem}{section}
\numberwithin{remark}{section}
\numberwithin{definition}{section}
\numberwithin{assumption}{section}

\DeclareMathOperator{\Ker}{Ker}
\newcommand{\R}{\mathbb{R}}

\title{Second order splitting of a class of fourth order PDEs with point constraints}
\author{Charles M. Elliott}
\address{Mathematics Institute, Zeeman Building, University of Warwick, Coventry. CV4 7AL. UK}
\email{C.M.Elliott@warwick.ac.uk}
\thanks{The work of CME was partially supported by the Royal Society via a Wolfson Research Merit Award.}

\author{Philip J. Herbert}
\address{Mathematics Institute, Zeeman Building, University of Warwick, Coventry. CV4 7AL. UK}
\email{P.J.Herbert@warwick.ac.uk}
\thanks{The research of PJH was funded by the Engineering and Physical Sciences Research Council grant  EP/H023364/1
 under the MASDOC centre for doctoral training at the University of Warwick.}
 
\date{}

\begin{document}

\begin{abstract}
{We formulate a well-posedness and approximation theory for a class of generalised saddle point problems with a specific form of constraints.  In this way we develop an approach to a class of fourth order elliptic partial differential equations with point constraints  using  the idea of splitting into coupled second order equations. An approach is formulated using a penalty method to impose the constraints. 
Our main motivation is to treat certain fourth order equations involving the biharmonic operator and point Dirichlet constraints for example  arising in the modelling of biomembranes on curved and flat surfaces  but the approach may be applied more generally. The theory for well-posedness and approximation is presented 
 in an abstract setting. Several examples are described together with some numerical experiments.}
\end{abstract}
\maketitle

\section{Introduction}
We study the well-posedness and approximation of a saddle point problem posed  in reflexive Banach spaces with a constraint in a Hilbert space.
Let $X,\,Y$ be reflexive Banach spaces, $X_0\subset X$ be a linear subspace and $S$ a Hilbert space with $T\colon X \to S$ being a given linear map.
	The problem we are interested in is:-
	
Given $(f,g,s)\in X^*\times Y^*\times S$ find $(u,w) \in X\times Y$ such that
\begin{equation}\label{eq:FirstStatement}
	\begin{split}
		c(u,\eta) + b(\eta,w) &= \langle f,\eta \rangle \quad \forall \eta \in X_0,
\\
		b(u,\xi) - m(w,\xi) &= \langle g, \xi \rangle \quad \forall \xi \in Y,
\\
		(Tu,z)_S &= (s,z)_S \quad \forall z \in S,
	\end{split}
\end{equation}
where $c(\cdot,\cdot),\, b(\cdot,\cdot)$ and $m(\cdot,\cdot)$ are bilinear forms and   precise assumptions  will be given in Section \ref{sectionAbstractForm}.

We approximate \eqref{eq:FirstStatement} by penalising the condition $(Tu-s,z)_S=0$, rather than imposing it.
This results in the problem:-

	Given $(f,g,s)\in X^*\times Y^*\times S$ and $\epsilon >0$,  find $(u^\epsilon,w^\epsilon) \in X \times Y$ such that
	\begin{align*}
	c(u^\epsilon,\eta) + \frac{1}{\epsilon}(T u^\epsilon,T \eta)_S
	+
	b(\eta, w^\epsilon)
	&=
	\langle f,\eta \rangle + \frac{1}{\epsilon}(s,T\eta)_S ~ &\forall \eta \in X,
	\\
	b(u^\epsilon,\xi) - m(w^\epsilon,\xi) &= \langle g,\xi\rangle ~ &\forall \xi \in Y.
	\end{align*}

Our abstract formulation is  motivated by applications of this theory to fourth order boundary value  
problems arising in the modelling of biomembranes posed on a flat domain, sphere or torus, with a specific  
example focusing on the sphere for ease of exposition. The  problems  are derived in 
\cite{ EllGraHob16,EllFriHob17, GraKie18, EllHatHer19} as approximations of minimisers of the 
Helfrich energy \cite{Hel73} with point constraints.
In this context these arise as Dirichlet constraints on the  membrane deformation modelling the 
attachment of point particles  to  the membrane at fixed locations.  
In the work of \cite{ButNaz11}, the authors consider an optimisation problem associated with 
bilaplace equation  with point Dirichlet conditions on a flat domain, $\Omega$. 

We have in mind the following setting. Let $\Gamma$ be a curved or flat $C^2$ two 
dimensional hypersurface with or without a boundary,  for $\infty>q>2>p>1$, 
set $$X=\left\{\eta \in W^{1,q}(\Gamma) \,:\, \int_\Gamma\eta =0 \right\},~
	Y=\left\{\xi \in W^{1,p}(\Gamma) \,:\, \int_\Gamma\xi=0\right\},~
L=L^2(\Gamma)$$and $S = \mathbb R^N$,  with bilinear forms,
$c\colon W^{1,q}(\Gamma) \times W^{1,q}(\Gamma) \to \R,\,
	b\colon W^{1,q}(\Gamma) \times W^{1,p}(\Gamma) \to \R$ and 
	$m\colon L^{2}(\Gamma) \times L^{2}(\Gamma) \to \R$.
In particular we have in mind an example of the form
\begin{align*}
	c(u,\eta)=
	\int_\Gamma \left( c_0 \nabla_\Gamma u \cdot \nabla_\Gamma \eta +c_1 u\eta\right),
	~~b(\eta,\xi)=
	\int_\Gamma \nabla_\Gamma \eta \cdot \nabla_\Gamma \xi + \eta\xi,~~
	m(\eta,\xi) &= \int_\Gamma \eta\xi
\end{align*}
where $c_0,\,c_1$ are bounded  but $c(\cdot,\cdot)$ is not coercive and a linear map 
$T\colon W^{1,q}(\Gamma)\rightarrow \mathbb R^N$ defined by $(T\eta)_j:=\eta(X_j),\,j=1,...,N$ with $X_j\in \Gamma,\,j=1,...,N$.



\subsection{Background}The study of saddle point problems is well documented, \cite{BofBreFor13, ErnGue13}, with 
many applications, for example in  fluid mechanics, \cite{GirRav86}, or in  linear elasticity, \cite{ArnFalWin07}. 
Note that in many of the cases in  which  $m\neq0$, the authors require some strong assumptions on $c$, at least positive 
semi definite, see \cite{BofBreFor13,CiaHuaZou03,KelLiu96}. 
The system \eqref{eq:FirstStatement} is an extension  to that considered in \cite{EllFriHob19}. 
 The  extended system  is posed  over an affine subspace of $X\times Y$, rather than over the whole space. 
 If in \eqref{eq:FirstStatement}, we were to seach for a solution in $X_0 \times Y$, 
 the first equation were to be considered with test functions in $X_0$ and the third equation to be dropped, 
 this recovers the abstract system studied in \cite{EllFriHob19}. 
We will use the assumptions made in \cite{EllFriHob19} together with  an additional assumption to handle the constraint. 
The assumptions will be given in Section \ref{sectionAbstractForm}. 
 
In \cite{CarKri82}, the authors consider the approximation of Stokes flow by penalising the incompressibility condition.
In particular, they show that the penalty terms approximate the pressure. We also consider 
an abstract problem with penalty and  show that, in our setting, the penalty terms converge to the 
Lagrange multiplier associated to the constraints.
Further to this, we show estimates between the solution to the problem with penalised constraint and the solution to the problem with enforced constraint.
 An abstract finite element theory with error bounds  is presented.  The results of this paper extend those of \cite{EllFriHob19} where for example, in \cite[Section 6 and 7]{EllFriHob19}  it is shown that the well posedness theory in that paper may be applied to a problem with penalised point constraints without consideration of the convergence with respect to the penalty parameter.

  The motivation for the abstract setting is to handle second order splitting for a class of fourth order surface PDEs with point Dirichlet constraints arising in the modelling of biomembranes, \cite{EllFriHob17}. The setting of \cite{EllFriHob19} may be directly applied to the penalty approximation for fixed penalty parameter but does not handle the hard constraint case. Here we show that the abstract setting is applicable and that the results apply to a surface finite element approximation using $H^1$ conforming surface finite elements. We also provide numerical experiments for this point constraint problem, considering both the grid refinements and refinements in penalty.  


\subsection{Outline of paper}
In Section \ref{sectionAbstractForm} we define the abstract saddle point system with constraint and with penalty, 
consisting of the bilinear forms $c,\,b,\,m$ and the inner product on the space of constraints.
Well-posedness for the penalty problem trivially follows from the results of \cite{EllFriHob19}.
Well-posedness for the constrained problem requires additional conditions, which are natural to the standard saddle point formulation.
We then show that under a set of assumptions which guarentee the problems are well posed, one obtains strong 
convergence with error estimates depending on the penalty parameter in the natural spaces.
An abstract finite element method is then discussed in Section \ref{Sec:AbstractFE}.
The explicit setting for the finite element method we will choose to model the application is given in Section \ref{Sec:SurfaceApplication}.
We conclude in Section \ref{Sec:Experiments} with some experimental examples which verify the proved convergence rates both in terms of the grid size and penalty parameter.

\section{Abstract problem}\label{sectionAbstractForm}
\subsection{Setting and problem formulation} We first define the spaces and functionals used along with 
the required assumptions. Throughout   $X$, $Y$ are reflexive Banach spaces, $L$  is a Hilbert space 
with $Y\subset L$ continuously embedded and $S$ is a separable Hilbert space with inner product $(\cdot,\cdot)_S$.
\begin{definition}\label{definitions}

Define the following
\begin{align*}
	c&\colon X \times X \to \R, \text{ bounded and bilinear},
	\\
	b&\colon X \times Y \to \R, \text{ bounded and bilinear},
	\\
	m&\colon L \times L \to \R, \text{ bounded, bilinear, symmetric and coercive},
	\\
	T&\colon X \to S, \text{ bounded, surjective and linear}.
\end{align*}
Let $s\in S$, define,
\[
	X_s := \{ x \in X : Tx = s\}.
\]
\end{definition}
It is clear that $X_s$ is non-empty by surjectivity of $T$.
With this general setting in mind, we formulate 
 a Lagrange multiplier problem and associated approximating penalised problem that we wish to consider.
\begin{problem}\label{lagrangeProblem}
Given $f\in X^*$, $g \in Y^*$ and $s \in S$, find $(u,w,\lambda) \in X\times Y \times S$ such that
\begin{align*}
	c(u,\eta) + b(\eta,w) + (T\eta,\lambda)_S &= \langle f, \eta \rangle \, &\forall \eta \in X,
	\\
	b(u,\xi) - m(w, \xi ) &= \langle g, \xi \rangle \, &\forall \xi \in Y,
	\\
	(Tu,z)_S &= (s,z)_S \, &\forall z \in S.
\end{align*}
\end{problem}

\begin{problem}
\label{penaltyProblem}
	Given $f\in X^*$, $g \in Y^*$, $s \in S$ and $\epsilon >0$,
	find $(u^\epsilon,w^\epsilon) \in X \times Y$ such that
	\begin{align*}
	c(u^\epsilon,\eta) + \frac{1}{\epsilon}(T u^\epsilon,T \eta)_S
	+
	b(\eta, w^\epsilon)
	&=
	\langle f,\eta \rangle + \frac{1}{\epsilon}(s,T\eta)_S \, &\forall& \eta \in X,
	\\
	b(u^\epsilon,\xi) - m(w^\epsilon,\xi) &= \langle g,\xi\rangle \, &\forall& \xi \in Y.
	\end{align*}
\end{problem}

\begin{remark}\label{hardProblem}
Observe that Problem \ref{lagrangeProblem} is equivalent to:-
Given $f\in X^*$, $g \in Y^*$ and $s \in S$, find $(u,w) \in X_s \times Y$ such that
\begin{align*}
	c(u,\eta)+b(\eta, w)
	&=
	\langle f,\eta \rangle \, &\forall \eta \in X_0,
	\\
	b(u,\xi) - m(w,\xi) &= \langle g,\xi\rangle \, &\forall \xi \in Y.
	\end{align*}
\end{remark}


We note that the assumptions we will make for the well-posedness of these two  abstract problems  differ. The following assumption is required for both the Lagrange multiplier problem and the problem with penalty.

\begin{assumption}
\label{ass:AbstractCoercivity}
There is $C>0$ and $\epsilon_0>0$ such that for all $(u,w)\in X\times Y$
\begin{equation}\label{eq:AbstractCoercivity}
	b(u,\xi) = m(w,\xi) ~~ \forall \xi \in Y \implies C\|w\|^2_L \leq c(u,u) +\frac{1}{\epsilon_0}(Tu,Tu)_S + m(w,w).
\end{equation}
\end{assumption}

\subsection{Well posedness of Lagrange multiplier problem}
\begin{assumption}\label{kernelInfSup}
There is $\zeta >0$ such that for any $(u,w) \in X_0\times Y$, 
\begin{equation}	\label{eq:BigInfSup}
	\zeta(\|u\|_X + \|w\|_Y )
	\leq
	\sup_{(\eta,\xi) \in X_0\times Y}
	\frac{c(u,\eta) + b(\eta,w) + b(u,\xi) - m(w,\xi)}{\|\eta\|_X + \|\xi\|_Y}.
\end{equation}

\end{assumption}
The assumption in \eqref{eq:BigInfSup} arises naturally from the standard saddle point problem:-

Find $x\in X\times Y$ and $p \in S$ such that
\begin{align*}
	a(x,y) + d(y,p) &= \langle\tilde{f},y\rangle~~\forall y \in X\times Y
	\\
	d(x,z) &= \langle s, z \rangle ~~\forall z \in S,
\end{align*}
where $x = (x_1,x_2)$ with $x_1 \in X,\, x_2 \in Y$, $a(x,y):= c(x_1,y_1) + b(y_1,x_2) + b(x_1,y_2) - m(x_2,y_2)$, $\langle\tilde{f},y \rangle = \langle f, y_1 \rangle + \langle g,y_2\rangle$ and $d(x,z) := (Tx_1, z)_S$, which is an equivalent formulation of Problem \ref{lagrangeProblem}.

\begin{theorem}\label{ThmLagrangeProblem}
	Given Assumptions \ref{ass:AbstractCoercivity} and \ref{kernelInfSup} hold, there a is unique solution to Problem \ref{lagrangeProblem}.
	Furthermore, it holds that there is $C>0$ such that,
\[
	\|u\|_X + \|w\|_Y + \| \lambda\|_S \leq C(\|f\|_{X^*} + \|g\|_{Y^*} + \| s \|_S).
\]
\end{theorem}
\begin{proof}
The existence and uniqueness of Problem \ref{lagrangeProblem} is a simple consequence of Assumption \ref{kernelInfSup} and surjectivity of $T$, using a standard theorem on saddle point problems, see \cite{ErnGue13} for example.
\end{proof}

\subsection{Well posedness of penalty approximation}
We will require the following assumptions on $b$, $X$ and $Y$, as in \cite{EllFriHob17}.
\begin{assumption}\label{InfSupSOSCoercivity}
There exist $\gamma,\beta>0$ such that
\begin{equation} \label{eq:AbstractInf-Sup}
	\beta \|\eta\|_X \leq \sup_{\xi \in Y} \frac{b(\eta,\xi)}{\|\xi\|_Y}~\forall \eta \in X
	~\mbox{and}~
	\gamma \|\xi\|_Y \leq \sup_{\eta \in X} \frac{b(\eta,\xi)}{\|\eta\|_X}~\forall \xi \in Y.
\end{equation}
\end{assumption}
In addition to the above, we also require there to be sufficiently well behaved approximating spaces. This allows for a Galerkin approximation.
We will see that we may pick finite element spaces satisfying the conditions.
\begin{assumption}\label{ass:discreteInfSupB}
There are finite dimensional approximating spaces $X_n \subset X$ and $Y_n \subset Y$, that is, $\forall (\eta,\xi) \in X\times Y$ there are $(\eta_n,\xi_n) \in X_n\times Y_n$ with $\|\eta-\eta_n\|_X + \|\xi-\xi_n\|_Y \to 0$.
We additionally assume discrete inf-sup conditions.
That is there are $\tilde{\beta},\tilde{\gamma}>0$, independent of $n$, such that
\begin{equation}\label{eq:DiscreteAbstractInf-Sup}
	\tilde\beta \|\eta_n\|_X \leq \sup_{\xi_n \in Y_n} \frac{b(\eta_n,\xi_n)}{\|\xi_n\|_Y}~\forall \eta_n \in X_n
	~\mbox{and}~
	\tilde\gamma \|\xi_n\|_Y \leq \sup_{\eta_n \in X_n} \frac{b(\eta_n,\xi_n)}{\|\eta_n\|_X}~ \forall \xi_n \in Y_n.
\end{equation}
We also assume that there is an interpolation map $I_n\colon Y \to Y_n$ for each $n$ such that
\[
	b(\eta_n,I_n\xi) = b(\eta_n,\xi)~ \forall(\eta_n,\xi) \in X_n \times Y,
\]
\[
	\sup_{\xi \in Y} \frac{\| \xi - I_n \xi\|_L}{\|\xi\|_Y} \to 0 ~\mbox{as}~ n \to \infty.
\]
\end{assumption}
We now quote two results which will be useful to refer to throughout the work, they may be found in \cite[Lemma 2.1 and 2.2]{EllFriHob19}.
\begin{lemma}\label{lemma:discreteInverse}
	Let Assumptions \ref{ass:AbstractCoercivity} and \ref{ass:discreteInfSupB}  hold.
	There is a linear map $G_n \colon Y^* \to X_n$ such that for any $\theta \in Y^*$
\[
	b(G_n \theta, \xi_n) = \langle \theta, \xi_h\rangle ~~ \forall \xi_n \in Y_n.
\]
\end{lemma}
\begin{lemma}\label{lemma:discreteCoercivity}
Let Assumptions \ref{ass:AbstractCoercivity} and \ref{ass:discreteInfSupB}  hold. There is $C,N>0$ such that for all $n\geq N$ and any $v_n \in Y_n$,
\begin{equation}\label{eq:DiscreteAbstractCoercivity}
	C\|v_n\|_L^2 \leq c(G_n(m(v_n,\cdot)),G_n(m(v_n,\cdot))) + \frac{1}{\epsilon_0}(T(G_n(m(v_n,\cdot))),T(G_n(m(v_n,\cdot))))_S + m(v_n,v_n).
\end{equation}
\end{lemma}
Now we have the required assumptions, we assert the well-posedness of Problem \ref{penaltyProblem}.
\begin{theorem}\label{Thm:PenaltyProblem}
Given Assumptions \ref{ass:AbstractCoercivity}, \ref{InfSupSOSCoercivity} and \ref{ass:discreteInfSupB}, there is a unique solution to Problem \ref{penaltyProblem}.
Furthermore it holds that
\[
	\|u^\epsilon\|_X + \|w^\epsilon\|_Y \leq C\left(1+\epsilon^{-1}\right) \left(\|f\|_{X^*} + \left(1+\epsilon^{-1} \right)\|g\|_{Y^*} + \epsilon^{-1}\|s\|_S \right)
\]
\end{theorem}
\begin{proof}
The existence and uniqueness follows from \cite[Theorem 2.2]{EllFriHob19}, with the estimate following from carrying through the $\epsilon$ terms. 
\end{proof}

Recall that we are interested in the case $\epsilon \to 0$, clearly in the above estimate, the bound diverges.
For numerics, one might like to take $\epsilon$ to be a function of the grid size, as such, the bounds diverging in $\epsilon$ means one must be restrictive in the relationship between grid size and $\epsilon$.
To show uniform bounds, we make use of the solution to Problem \ref{lagrangeProblem}.
\subsection{Convergence of penalty approximation}  In this subsection we  assume  Assumptions \ref{ass:AbstractCoercivity}, \ref{kernelInfSup}, \ref{InfSupSOSCoercivity}, and \ref{ass:discreteInfSupB} to hold.
\begin{proposition}\label{Prop:FirstAbstractEstimate} 
Let $(u,w,\lambda)$ solve Problem \ref{lagrangeProblem} and $(u^\epsilon,w^\epsilon)$ solve Problem \ref{penaltyProblem}.
Then it holds that there is $C>0$, independent of $\epsilon$ such that
\[
	\|w-w^\epsilon\|_Y + \|u-u^\epsilon\|_X + \left\lVert \lambda - \epsilon^{-1}T(u^\epsilon-u) \right\rVert_S \leq C \sqrt{\epsilon} \| \lambda \|_S.
\]
\end{proposition}

\begin{proof}
From \eqref{eq:AbstractInf-Sup}, it holds,
\[
	\beta \|u-u^\epsilon\|_X
	\leq
	\sup_{\xi \in Y} \frac{b(u-u^\epsilon,\xi)}{\|\xi\|_Y}
	=
	\sup_{\xi \in Y} \frac{m(w-w^\epsilon,\xi)}{\|\xi\|_Y}
	\leq
	C \|w-w^\epsilon\|_L,
\]
where we have used the second equations of the systems.
Now by taking differences of the first equations of the systems,
	\[
		c(u^\epsilon - u,\eta) + b(\eta, w^\epsilon-w) + \frac{1}{\epsilon} (T(u^\epsilon-u),T\eta)_S = (\lambda,T\eta)_S\, \forall \eta \in X.
	\]
By letting $\eta = u^\epsilon-u$ in the above and from \eqref{eq:AbstractCoercivity}, one has,
\begin{align*}
	C\|w-w^\epsilon\|_L^2 +  \left(\frac{1}{\epsilon}- \frac{1}{\epsilon_0} \right)\|T(u-u^\epsilon)\|_S^2
	\leq&
	c(u^\epsilon-u,u^\epsilon - u) + m(w^\epsilon-w,w^\epsilon-w)
+ \frac{1}{\epsilon}(T(u-u^\epsilon),T(u-u^\epsilon))_S
	\\
	=& (\lambda, T(u^\epsilon-u))_S
	\leq C \| \lambda \|_S \|T(u^\epsilon -u)\|_S
	\\
	\leq&
	C\frac{\epsilon}{2} \| \lambda \|_S^2 + \frac{1}{2\epsilon} \|T(u^\epsilon -u)\|_S^2,
\end{align*}
this has shown
\begin{equation}\label{eq:IntermediateAbstractEpsilonEstimate}
	\|w-w^\epsilon\|_L + \|u-u^\epsilon\|_X + \frac{1}{\sqrt{\epsilon}}\|T(u-u^\epsilon)\|_S \leq C \sqrt{\epsilon} \| \lambda \|_S,
\end{equation}
which shows the result for the $\|u-u^\epsilon\|_X$.

%

	Due to Assumption \ref{kernelInfSup}, one has,
\[
	\kappa\|w-w^\epsilon\|_Y \leq \sup_{(\eta,\xi)\in X_0\times Y} \frac{c(0,\eta) +b(\eta,w-w^\epsilon) + b(0,\xi) - m(w-w^\epsilon,\xi)}{\|\eta\|_X + \|\xi\|_Y}.
\]
One then calculates
\[
	c(u^\epsilon-u,\eta)+b(\eta,w^\epsilon-w) = 0~~\forall \eta \in X_0,
\]
hence
\[
	\kappa \|w-w^\epsilon\|_Y\leq \sup_{(\eta,\xi)\in X_0\times Y} \frac{c(u^\epsilon-u,\eta)- m(w-w^\epsilon,\xi)}{\|\eta\|_X + \|\xi\|_Y} \leq C ( \|u^\epsilon-u\|_X + \|w-w^\epsilon\|_L)
\]
this proves the result for $\|w-w^\epsilon\|_Y$ when making use of \eqref{eq:IntermediateAbstractEpsilonEstimate}.
Finally, from surjectivity of $T$, one has
\[
	\alpha  \left\lVert \lambda - {\epsilon}^{-1}T(u^\epsilon-u)\right\rVert_S \leq \sup_{\eta \in X} \frac{ (\lambda - {\epsilon}^{-1}T(u^\epsilon-u), T\eta)_S}{\|\eta\|_X},
\]
where again one calculates
\[
	(\lambda - {\epsilon}^{-1}T(u^\epsilon-u),T\eta)_S = c(u^\epsilon-u, \eta) + b(\eta, w^\epsilon-w)~~ \forall \eta \in X,
\]
thus
\[
	\alpha  \left\lVert\lambda - {\epsilon}^{-1}T(u^\epsilon-u)\right\rVert_S \leq C\|\lambda\|_S \sqrt{\epsilon}.
\]
\end{proof}
Indeed it is possible to use an Aubin-Nitsche type argument to give a higher rate of convergence for $\|u-u^\epsilon\|_X$ and $\|w-w^\epsilon\|_L$.
\begin{proposition}\label{prop:increasedEpsilonConvergence}
	Let $(u,w,\lambda)$ solve Problem \ref{lagrangeProblem}, let $(u^\epsilon,w^\epsilon)$ solve Problem \ref{penaltyProblem}.
	Then there is a $C>0$ independent of $\epsilon$ such that
	\[
		\|u-u^\epsilon\|_X + \|w-w^\epsilon\|_L \leq C \epsilon.
	\]
\end{proposition}
\begin{proof}
	Let $(\psi,\phi,\chi)\in X \times Y \times S$ satisfy
	\begin{equation}\label{eq:PenaltyAubinNitsche}\begin{split}
		c(\eta,\psi) + b(\eta,\phi) + (\chi,T\eta)_S &=0~
		\\
		b(\psi,\xi) - m(\phi,\xi) &=( w-w^\epsilon,\xi)_L ~
		\\
		(T\psi,z) &=0~
	\end{split}
	\begin{split}
	\forall \eta \in X,
	\\
	\forall \xi \in Y,
	\\
	\forall z \in S.
	\end{split}\end{equation}
	This exists and is unique by Theorem \ref{ThmLagrangeProblem} and it holds that there is some $C>0$ with
	\[
		\|\psi\|_X + \|\phi\|_Y + \|\chi\|_S \leq C \|w-w^\epsilon\|_L.
	\]
	Now test \eqref{eq:PenaltyAubinNitsche} with $\left( u-u^\epsilon,w-w^\epsilon,T(u-u^\epsilon) \right)$ and sum the first two equations,
	\begin{align*}
	\|w-w^\epsilon\|^2_L
	=&
	c(u-u^\epsilon,\psi) + b(u-u^\epsilon,\phi) + (\chi,T(u-u^\epsilon))_S + b(\psi,w-w^\epsilon) - m(w-w^\epsilon,\phi)
	\\
	=&
	(\chi,T(u-u^\epsilon))_S.
	\end{align*}
	Where one has $b(u-u^\epsilon,\phi) = m(w-w^\epsilon,\phi)$ and $c(u-u^\epsilon,\psi) = -b(\psi,w-w^\epsilon)$ with the second following from $(T\psi,\lambda)_S = \epsilon^{-1}(T(u^\epsilon-u),T\psi)_S= 0$.
	Hence from the estimates for $\|T(u-u^\epsilon)\|_S$ shown in Proposition \ref{Prop:FirstAbstractEstimate} with $\|u-u^\epsilon\|_X \leq C \|w-w^\epsilon\|_L$ and $\|\chi\|_S \leq C \|w-w^\epsilon\|_L$, the result.
\end{proof}

We may now use Proposition \ref{Prop:FirstAbstractEstimate} to give uniform bounds on $u^\epsilon$ and $w^\epsilon$.
\begin{corollary}\label{cor:uniformBoundsEpsi}
	Given $\epsilon>0$ sufficiently small, let $(u^\epsilon,w^\epsilon) \in X \times Y$ solve Problem \ref{penaltyProblem}.
	Then there is $C>0$ with
	\[
		\|u^\epsilon\|_X + \|w^\epsilon\|_Y
		+
		\frac{1}{\epsilon}\left\lVert T(u^\epsilon-u)\right\rVert_S
		\leq C (\|f\|_{X^*} + \|g\|_{Y^*} + \|s\|_S).
	\]
\end{corollary}


\section{Abstract finite element method}\label{Sec:AbstractFE}
We now formulate and analyse an abstract finite element method to approximate the solutions to Problems \ref{penaltyProblem} 
and Problem \ref{lagrangeProblem}.
We formulate the method in the sense of an external approximation, this is motivated by our wish to apply the formulation to surface finite elements.
\begin{definition}
	For $h>0$, let $X_h$, $Y_h$ and $S_h$ be finite dimensional normed vector spaces with lift operators
\[
	l^X_h\colon X_h \to X,~~ l^Y_h\colon Y_h \to Y,~~ l^S_h\colon S_h \to S,
\]
which are bounded, linear and injective and define $X_h^l :=l^X_h(X_h),$ $Y_h^l:=l^Y_h(Y_h),$ $S_h^l:= l^S_h(S_h)$.

Let $c_h,\,b_h,\,m_h$ denote bilinear forms such that
\begin{align*}
	c_h\colon& X_h \times X_h \to \R,~
	b_h\colon X_h \times Y_h \to \R,~
	m_h\colon Y_h \times Y_h \to \R,
\end{align*}
with $m_h$ symmetric and let $(\cdot,\cdot)_{S_h}$ be an inner product on $S_h$.
Further, let $f_h \in X_h^*,$ $g_h\in Y_h^*$ and $s_h \in S_h$.
\end{definition}
The spaces $X_h^l$, $Y_h^l$ and $S_h^l$ may be identified as the spaces $X_n$, $Y_n$ and $S_n$ in Section \ref{sectionAbstractForm}.
For this section, we assume that the following approximations hold.
\begin{assumption}\label{ass:AbstractApproximationProperties}
	We assume that there is $C>0$ and $k\in \mathbb{N}$ such that
\begin{align*}
	|c(\eta_h^l,\xi_h^l) - c_h(\eta_h,\xi_h)| \leq& Ch^k \|\eta_h^l\|_X \|\xi_h^l\|_X\quad\forall (\eta_h,\xi_h) \in X_h\times X_h,
\\	
	|b(\eta_h^l,\xi_h^l) - b_h(\eta_h,\xi_h)| \leq& Ch^k \|\eta_h^l\|_X \|\xi_h^l\|_Y\quad\forall (\eta_h,\xi_h) \in X_h\times Y_h,
\\
	|m(\eta_h^l,\xi_h^l) - m_h(\eta_h,\xi_h)| \leq& Ch^k \|\eta_h^l\|_L \|\xi_h^l\|_L\quad\forall (\eta_h,\xi_h) \in Y_h\times Y_h,
	\\
	|(z_h^l, T\eta_h^l)_{S} - (z_h,T_h \eta_h)_{S_h} | \leq& Ch^k \|\eta_h^l\|_X \|z_h^l\|_S\quad\forall (\eta_h,z_h)\in X_h \times S_h,
	\\
	|\langle f, \eta_h^l \rangle - \langle f_h,\eta_h\rangle| \leq& Ch^k \|f\|_{X^*} \|\eta_h^l\|_X\quad\forall \eta_h\in X_h,
	\\
	|\langle g, \xi_h^l \rangle - \langle g_h,\xi_h\rangle| \leq& Ch^k \|g\|_{Y^*} \|\xi_h^l\|_Y\quad\forall \xi_h\in Y_h,
	\\
	|(s,z_h^l)_S -(s_h,z_h)_{S_h}| \leq& Ch^k \|s\|_S\|z_h^l\|_S\quad\forall z_h \in S_h.
\end{align*}
\end{assumption}

The finite element approximations can now be formulated.

\subsection{Finite element method for the Lagrange multiplier problem}
For this subsection, we suppose Assumptions \ref{ass:AbstractCoercivity}, \ref{kernelInfSup}, \ref{ass:discreteInfSupB} and the following Assumption \ref{ass:DiscreteLagrange} hold true.
\begin{assumption}\label{ass:DiscreteLagrange}
There is $\tilde{\zeta}>0$ independent of $h$ such that for any $(w_h,\chi_h) \in Y_h \times S_h$,
\[
	\tilde{\zeta} (\|\chi_h^l\|_S + \|w_h^l\|_Y)
	\leq
	\sup_{(\eta_h,\xi_h)\in X_h\times Y_h}\frac{b(\eta_h^l,w_h^l)+(T\eta_h^l,\chi_h^l)_S + m(w_h^l,\xi_h^l)}{\|\eta_h^l\|_X+\|\xi_h^l\|_Y}.
\]
\end{assumption}
\begin{problem}\label{discreteLagrangeProblem}
Find $(u_h,w_h,\lambda_h) \in X_h\times Y_h\times S_h$ such that
\begin{align*}
	c_h(u_h,\eta_h) + b_h(\eta_h,w_h) + (\lambda_h,T_h\eta_h)_{S_h} &= \langle f_h,\eta_h\rangle\quad \forall \eta_h \in X_h,
	\\
	b_h(u_h,\xi_h) - m_h(w_h,\xi_h) &= \langle g_h, \xi_h \rangle \quad \forall \xi_h \in Y_h
	\\
	(T_h u_h,z_h)_{S_h} &= (s_h,z_h)_{S_h}\quad \forall z_h \in S_h.
\end{align*}
\end{problem}

\begin{theorem}\label{ThmDiscreteLagrangeProblem}
For sufficiently small $h$, there exists a solution to Problem \ref{discreteLagrangeProblem}.
Furthermore there is $C>0$ independent of $h$ such that
\begin{align*}
	\|u-u_h^l\|_X + \| w-w_h^l\|_Y + \|\lambda-\lambda_h^l\|_S
	\leq
	C\inf_{\substack{(\eta_h,\xi_h,\chi_h) \\ \in X_h\times Y_h\times S_h}} &\left( \|u-\eta_h^l\|_X + \|w-\xi_h^l\|_Y + \|\lambda - \chi_h^l\|_S \right)
	\\
	&+
	Ch^k \left( \|f\|_{X^*} + \|g\|_{Y^*} + \|s\|_S\right).
\end{align*}
\end{theorem}
\begin{proof}
	The argument follows along the lines of \cite[Theorem 3.1]{EllFriHob19}.
	For existence and uniqueness, it is sufficient to show uniqueness for the homogeneous case, $f_h=g_h=s_h=0$, as the system is linear and finite dimensional.
	In this case, we see that
	\begin{equation}\label{eq:discreteLagrangeHomongenous}
		c_h(u_h,u_h)+m_h(w_h,w_h) =0 ~\mbox{and}~ (T_hu_h, T_h u_h)_{S_h} =0.
	\end{equation}
	We write $G_h^l\colon Y^* \to X_h^l$ to be, for each $y \in Y^*$, $G_h^l y$ is the unique element such that
	\[
		b (G_h^l y, \xi_h^l) = \langle y,\xi_h^l\rangle\quad \forall \xi_h^l \in Y_h^l,
	\]
	as considered in Lemma \ref{lemma:discreteInverse}.
	Now notice that in this homogeneous case,
	\begin{align*}
		\tilde{\beta} \|u_h^l - G_h^l m(w_h^l,\cdot)\|_X
		\leq&
		\sup_{\xi_h\in Y_h} \frac{b(u_h^l-G_h^l m(w_h^l,\cdot),\xi_h^l)}{\|\xi_h^l\|_Y}
		\\
		\leq&
		\sup_{\xi_h\in Y_h} \frac{b(u_h^l,\xi_h^l) - b_h(u_h,\xi_h) + m_h(w_h,\xi_h) - m(w_h^l,\xi_h) }{\|\xi_h^l\|_Y}
		\\
		\leq&
		Ch^k \|w_h^l\|_L,
	\end{align*}
	where the final line follows from Assumption \ref{ass:AbstractApproximationProperties} and that $\|u_h^l\|_X \leq C \|w_h^l\|_L$, which is a consquence of the second equation of the system with the discrete inf-sup \eqref{eq:DiscreteAbstractInf-Sup}.
	It follows from \eqref{eq:DiscreteAbstractCoercivity},
	\begin{align*}
	C\|w_h^l \|_L^2
	\leq& c(G_h^lm(w_h^l,\cdot),G_h^lm(w_h^l,\cdot)) + m(w_h^l,w_h^l)
	+ \frac{1}{\epsilon_0}(TG_h^lm(w_h^l,\cdot),TG_h^lm(w_h^l,\cdot))_{S}
	\\
	=&
	c(u_h^l, u_h^l) + m(w_h^l,w_h^l) +  \frac{1}{\epsilon_0}(T u_h^l,T u_h^l)_{S}  -c_h(u_h,u_h) -m_h(w_h,w_h) -\frac{1}{\epsilon_0}(T_h u_h,T_h u_h)_{S_h}
	\\
	&+
	c(G_h^lm(w_h^l,\cdot),G_h^lm(w_h^l,\cdot)) - c(u_h^l,u_h^l) + \frac{1}{\epsilon_0}(TG_h^lm(w_h^l,\cdot),TG_h^lm(w_h^l,\cdot))_{S} - \frac{1}{\epsilon_0}(T u_h^l,T u_h^l)_{S}
	\\
	\leq& \tilde{C}h^k \|w_h^l\|_L^2,
	\end{align*}
	where we have made use of the above bound, $\|u_h^l - G_h^l m(w_h^l,\cdot)\|_X \leq C h^k \|w_h^l\|_L$, along with the approximation assumptions on the discrete bilinear forms, Assumption \ref{ass:AbstractApproximationProperties}.
	Thus for sufficiently small $h$, $w_h^l = 0$, it follows that $\lambda_h = w_h = u_h =0$, where $u_h=0$ comes from the second equation and $\lambda_h=0$ follows from the first equation and Assumption \ref{ass:DiscreteLagrange}.
	
	We have for any $\eta_h \in X_h$ and $\xi_h \in Y_h$,
\begin{align*}
	\tilde{\beta} \|u_h^l-\eta_h^l\|_X
	\leq&
	\sup_{v_h \in Y_h} \frac{b(u_h^l-\eta_h^l,v_h^l)}{\|v_h^l\|_Y}
	\\
	\leq&
	C\left(\|u-\eta_h^l\|_X + \| w -\xi_h^l\|_Y + \|w_h^l - \xi_h^l \|_L\right)
	+
	Ch^k \left( \|u_h^l\|_X + \|g\|_{Y^*} + \|w_h^l\|_L \right),
\end{align*}
where we have made use of
\begin{align*}
	b(u_h^l-\eta_h^l,v_h^l) =& b(u-\eta_h^l,v_h^l) + b(u_h^l,v_h^l) - b_h(u_h,v_h) + \langle g_h,v_h\rangle
	-\langle g,v_h^l\rangle
	\\
	&+
	m_h(w_h,v_h) -m(w_h^l,v_h^l) + m(w_h^l-\xi_h^l,v_h^l) - m(w-\xi_h^l,v_h^l).
\end{align*}
In a similar fashion, for any $\chi_h \in S_h$, one obtains from Assumption \ref{ass:DiscreteLagrange},
\begin{align*}
	\tilde{\kappa}(\|w_h^l - \xi_h^l\|_Y + \|\lambda_h^l - \chi_h^l\|_S) \leq& \sup_{(x_h,v_h) \in X_h \times Y_h}
	\frac{b(x_h^l, w_h^l - \xi_h^l ) + (\lambda_h^l - \chi_h^l, Tx_h^l)_S + m(w_h^l-\xi_h^l, v_h^l)}{\|x_h^l\|_X +\|v_h^l\|_Y}
	\\
	\leq&
	C\left(\|u-\eta_h^l\|_X + \|w-\xi_h^l\|_Y + \|\lambda-\chi_h^l \|_S + \|u_h^l-\eta_h^l\|_X\right)
	\\
	&+  C h^k \left(\|u_h^l\|_X + \|w_h^l\|_Y + \|f\|_{X^*} + \| \lambda_h^l \|_S\right)
\end{align*}
where we have made use of
\begin{align*}
	b(x_h^l,w_h^l - \xi_h^l) +(\lambda_h^l - \chi_h^l,Tx_h^l)_S
	=&
	b(x_h^l, w-\xi_h^l) + (\lambda - \chi_h^l,Tx_h^l)_S + b(x_h^l, w_h^l) - \langle f,x_h^l\rangle
	\\ &+ c(u,x_h^l)
	+
	(\lambda_h^l,Tx_h^l)_S +\langle f_h,x_h\rangle -(\lambda_h,T_h x_h)_{S_h}
	\\
	&-b_h(x_h,w_h) -c_h(u_h,x_h)
	\\
	\leq&
	C\|x_h^l\|_X \left( \|w-\xi_h^l\|_Y + \|\lambda-\chi_h^l \|_S\right)
	\\
	&+
	Ch^k \|x_h^l\|_X \left( \|w_h^l\|_Y + \|f\|_{X^*}+ \|g\|_{Y^*} + \| \lambda_h^l \|_S  \right)
	\\
	&+ c(u-\eta_h^l,x_h^l) -c(u_h^l-\eta_h^l,x_h^l) + c(u_h^l,x_h^l) - c_h(u_h,x_h),
\end{align*}
and also
\begin{align*}
	m(w_h^l-\xi_h^l,v_h^l) =&
	m(w-\xi_h^l,v_h^l) + m(w_h^l,v_h^l) - b(u,v_h^l)
	\\
	&+\langle g,v_h^l \rangle
	-m_h(w_h,v_h) + b_h(u_h,v_h) -\langle g_h,v_h\rangle.
\end{align*}
Combining these two inequalities gives
\begin{equation}\label{EstimateWithLRHS}
\begin{split}
	\|u_h^l - \eta_h^l \|_X + \|w_h^l- \xi_h^l\|_Y + \|\lambda_h^l -\chi_h^l\|_S
	\leq&
	C\left( \|u-\eta_h^l\|_X + \|w-\xi_h^l\|_Y + \|\lambda-\chi_h^l\|_S + \|w_h^l -\xi_h^l\|_L\right)
	\\
	&+
	C
	h^k\left( \|u_h^l\|_X + \|w_h^l\|_Y + \|\lambda_h^l\|_S + \|f\|_{X^*} + \|g\|_{Y^*}\right).
\end{split}
\end{equation}
All that is remaining, is to bound the $L$-norm which appears on the right hand side.
We again use Lemma \ref{lemma:discreteCoercivity} to obtain,
\begin{equation}
	\begin{split}\label{eq:LagrangeExistenceFirstLEstimate}
	C\|w_h^l-\xi_h^l\|_L^2
	\leq&
	[c(u_h^l-\eta_h^l,u_h^l-\eta_h^l) + m(w_h^l-\xi_h^l,w_h^l-\xi_h^l)
	\\
	&-
	c_h(u_h-\eta_h,u_h-\eta_h) - m_h(w_h-\xi_h,w_h-\xi_h)]
	\\
	&+
	[c_h(u_h-\eta_h,u_h-\eta_h) + m_h(w_h-\xi_h,w_h-\xi_h)]
	\\
	&+
	[c(G_h^lm(w_h^l-\xi_h^l,\cdot),G_h^lm(w_h^l-\xi_h^l,\cdot))-c(u_h^l-\eta_h^l,u_h^l-\eta_h^l)]
	\\
	&+
	\frac{1}{\epsilon_0}(T((G_h^lm(w_h^l-\xi_h^l,\cdot)),T(G_h^lm(w_h^l-\xi_h^l,\cdot)))_S.
	\end{split}
\end{equation}
The first and third terms are dealt with in \cite{EllFriHob19} (the first term is dealt with in their (3.5) and the third term immediately following), giving
\begin{align*}
	|c(u_h^l-\eta_h^l,u_h^l-\eta_h^l) + m(w_h^l-\xi_h^l,w_h^l-\xi_h^l)
	-&
	c_h(u_h-\eta_h,u_h-\eta_h) - m_h(w_h-\xi_h,w_h-\xi_h)|
	\\
	&\leq
	Ch^k(\mathcal{B}^2 + \mathcal{B} \|w_h^l-\xi_h^l\|_L + \|w_h^l-\xi_h^l\|_L^2),
\end{align*}
this is a consequence of Assumption \ref{ass:AbstractApproximationProperties},
and
\[
	|c(G_h^lm(w_h^l-\xi_h^l,\cdot),G_h^lm(w_h^l-\xi_h^l,\cdot))-c(u_h^l-\eta_h^l,u_h^l-\eta_h^l)|
	\leq
	C (\mathcal{B}^2 + \mathcal{B}\|w_h^l-\xi_h^l\|_L),
\]
which follows from the discrete inf-sup, Assumption \ref{ass:discreteInfSupB} and the definition of $G_h^l$.
Where
\begin{align*}
	\mathcal{B}:=&
	\|u-\eta_h^l\|_X + \|w- \xi_h^l\|_Y + \|\lambda-\chi_h^l \|_S
	\\
	&+
	h^k \left(\|g\|_{Y^*} + \|f\|_{X^*} + \|s\|_S + \|u_h^l\|_X + \|w_h^l\|_Y + \|\lambda_h^l\|_S + \|\eta_h^l\|_X + \| \xi _h^l\|_Y + \|\chi_h^l\|_S \right).
\end{align*}
The fourth term of \eqref{eq:LagrangeExistenceFirstLEstimate} may be dealt with in the same way as the third and first part of the same equation.
For the remaining term of \eqref{eq:LagrangeExistenceFirstLEstimate}, we calculate,
\begin{equation}\label{eq:LagrangeExistenceSecondLEstimate}
\begin{split}
	c_h(u_h-\eta_h,u_h-\eta_h) +& m_h(w_h-\xi_h,w_h-\xi_h)
	\\
	=&
	[c_h(u_h-\eta_h,u_h-\eta_h) + b_h(u_h-\eta_h,w_h-\xi_h) + (\lambda_h-\chi_h,T_h(u_h-\eta_h))_{S_h}]
	\\&-
	[b_h(u_h-\eta_h,w_h-\xi_h) - m_h(w_h-\xi_h,w_h-\xi_h)]
	\\&-
	(\lambda_h-\chi_h,T_h(u_h-\eta_h))_{S_h}.
\end{split}
\end{equation}
We now split up the calculation of \eqref{eq:LagrangeExistenceSecondLEstimate}.
For the second term of \eqref{eq:LagrangeExistenceSecondLEstimate},
\begin{align*}
	b_h(u_h-\eta_h,w_h-\xi_h) -& m_h(w_h-\xi_h,w_h-\xi_h)
	\\	
	=&
	b(u-\eta_h^l,w_h^l-\xi_h^l) - m(w-\xi_h^l) - \langle g,w_h^l-\xi_h^l\rangle + \langle g_h, w_h-\xi_h\rangle
	\\
	&+
	b(\eta_h^l,w_h^l-\xi_h^l) - m(\xi_h^l,w_h^l-\xi_h^l) - b_h(\eta_h,w_h-\xi_h) + m_h(\xi_h,w_h-\xi_h)
	\\
	\leq&
	C\|w_h^l-\xi_h^l\|_Y \left( \|u-\eta_h^l\|_X + \|w-\xi_h^l\|_Y + h^k \left(\|g\|_{Y^*} + \|\eta_h^l\|_X + \|\xi_h^l\|_Y \right)\right)
	\\
	\leq&
	C (\mathcal{B}^2 + \mathcal{B}\|w_h^l-\xi_h^l\|_L).
\end{align*}
With the first term of \eqref{eq:LagrangeExistenceSecondLEstimate} resulting in,
\begin{align*}
	c_h(u_h-\eta_h,u_h-\eta_h) &+ b_h(u_h-\eta_h,w_h-\xi_h) + (\lambda_h-\chi_h, T_h (u_h-\eta_h))_{S_h}
	\\
	=&
	c(u-\eta_h^l,u_h^l-\eta_h^l) + b(u_h^l-\eta_h^l,w-\xi_h^l) + (\lambda-\chi_h^l,T(u_h^l-\eta_h^l))_S 
	\\
	&-
	\langle f, u_h^l - \eta_h^l \rangle + \langle f_h, u_h - \eta_h\rangle
	+ c(\eta_h^l,u_h^l-\eta_h^l) + b(u_h^l-\eta_h^l,\xi_h^l)
	\\
	&+
	(\chi_h^l,T(u_h^l-\eta_h^l))_S -
	c_h(\eta_h,u_h-\eta_h) - b_h(u_h-\eta_h,\xi_h) - (\chi_h, T_h (u_h-\eta_h))_{S_h}
	\\
	\leq&
	C\|u_h^l - \eta_h^l \|_X\left( \|u-\eta_h^l\|_X + \|w- \xi_h^l\|_Y + \|\lambda-\chi_h^l\|_S\right)
	\\
	&+
	Ch^k \|u_h^l-\eta_h^l\|_X\left( \|\eta_h^l\|_X + \|\xi_h^l\|_Y +\|\chi_h^l\|_S +\|f\|_{X^*} \right)
	\\
	\leq&C (\mathcal{B}^2 + \mathcal{B}\|w_h^l-\xi_h^l\|_L).
\end{align*}
Finally, from the third term of \eqref{eq:LagrangeExistenceSecondLEstimate}, one has
\begin{align*}
	(T_h(u_h-\eta_h), \lambda_h - \chi_h)_{S_h}
	=&
	(T(u-\eta_h^l),\lambda_h^l-\chi_h^l)_S + (T \eta_h^l,\lambda_h^l-\chi_h^l)_S - (s,\lambda_h^l - \chi_h^l)_S
	\\
	&- (T_h \eta_h, \lambda_h-\chi_h)_{S_h} + (s_h,\lambda_h-\chi_h)_{S_h}
	\\
	\leq&
	C\|\lambda_h^l-\chi_h^l\|_S(\|u-\eta_h^l\|_X + h^k(\|s\|_S + \|\eta_h^l\|_X))
	\\
	\leq&
	C (\mathcal{B}^2 + \mathcal{B}\|w_h^l-\xi_h^l\|_L).
\end{align*}
Thus for sufficiently small $h$, we have
\[
	\|w_h^l- \xi_h^l\|_L^2 \leq C (\mathcal{B}^2 + \mathcal{B}\|w_h^l-\xi_h^l\|_L),
\]
which after an application of Young's inequality, gives us the desired control on $\|w_h^l- \xi_h^l\|_L$.
When putting this into \eqref{EstimateWithLRHS} gives,
\begin{equation}\label{eq:LagrangeAbstractFEMBBound}
	\|u_h^l-\eta_h^l\|_X + \|w_h^l-\xi_h^l\|_Y + \|\lambda_h^l-\chi_h^l\|_S
	\leq
	C\mathcal{B}.
\end{equation}
By choice of $\eta_h = \xi_h=\chi_h=0$, one has
\begin{align*}
	\|u_h^l\|_X + \|w_h^l\|_Y + \|\lambda_h^l\|_S
	\leq&
	C\left(\|u\|_X + \|w\|_Y + \|\lambda\|_S\right)
	\\
	&+
	Ch^k\left( \|f\|_{X^*} + \|g\|_{Y^*} + \|s\|_S + \|u_h^l\|_X + \|w_h^l\|_Y + \|\lambda_h^l\|_S \right),
\end{align*}
which for sufficiently small $h$, and the estimates shown in Theorem \ref{ThmLagrangeProblem}, gives
\[
	\|u_h^l\|_X + \|w_h^l\|_Y + \|\lambda_h^l\|_S
	\leq
	C\left( \|f\|_{X^*} + \|g\|_{Y^*} + \|s\|_S \right).
\]
Using this and the triangle inequality gives
\begin{align*}
\|u-u_h^l\|_X + \|w-w_h^l\|_Y + \|\lambda-\lambda_h^l\|_S
\leq&
C\left(\|u-\eta_h^l\|_X + \|w-\xi_h^l\|_Y + \|\lambda-\chi_h^l\|_S\right)
\\
+&Ch^k\left( \|g\|_{Y^*} + \|f\|_{X^*} + \|s\|_S + \|\eta_h^l\|_X + \|\xi_h^l\|_Y + \|\chi_h^l\|_S \right).
\end{align*}
We are left to remove the isolated $\eta_h^l$, $\xi_h^l$ and $\chi_h^l$ terms on the right hand side, this can be done by
\[	
	\|\eta_h^l\|_X + \| \xi_h^l\|_Y + \|\chi_h^l\|_S
	\leq
	\|u-\eta_h^l\|_X + \|w-\xi_h^l\|_Y + \|\lambda - \chi_h^l\|_S + \|u\|_X + \|w\|_Y + \|\lambda\|_S.
\]
Thus for sufficiently small $h$, it holds
\begin{align*}
	\|u-u_h^l\|_X + \| w-w_h^l\|_Y + \|\lambda-\lambda_h^l\|_S
	\leq
	C \big[ \|u-\eta_h^l\|_X + \| w-\xi_h^l\|_Y + \|\lambda-\chi_h^l\|_S 
	\\
	+
	h^k\left(\|f\|_{X^*} + \|g\|_{Y^*} + \|s\|_S\right)\big].
\end{align*}
Taking the infimum over $(\eta_h,\xi_h,\chi_h) \in X_h\times Y_h\times S_h$ gives the result.
\end{proof}
In applications one may have interpolation operators which allow an error bound of the form of $Ch^\alpha$ for some $0\leq \alpha \leq k$.
The magnitude of $\alpha$ will depend on the regularity of the solution.
\begin{corollary}\label{Corollary:DiscreteLagrangeEstimates}
Suppose there are Banach spaces $\bar{X}$, $\bar{Y}$, $\bar{S}$ continuously embedded in $X$, $Y$, $S$ respectively with $(u,w,\lambda) \in \bar{X}\times\bar{Y}\times\bar{S}$.
Additionally assume there is $\bar{C},\alpha>0$ independent of $h$ such that
\[
	\inf_{(\eta_h,\xi_h,\chi_h)\in X_h \times Y_h\times S_h}
		\left( \|u - \eta_h^l\|_X + \|w-\xi_h^l\|_Y + \|\lambda - \chi_h^l\|_S \right)
		\leq
		\bar{C}h^\alpha(\|u\|_{\bar{X}} + \|w\|_{\bar{Y}} + \|\lambda\|_{\bar{S}}).
\]
Then for sufficiently small $h$, there is $C>0$ such that
\[
	\|u-u_h^l\|_X + \|w- w_h^l\|_Y + \|\lambda-\lambda_h^l\|_S
	\leq
	C h^{\min(\alpha,k)} (\|u\|_{\bar{X}} + \|w\|_{\bar{Y}} + \|\lambda\|_{\bar{S}} + \|f\|_{X^*} + \|g\|_{Y^*} + \|s\|_{S}).
\]
\end{corollary}
\begin{proposition}\label{prop:DiscreteLagrangeEstimates}
Under the assumptions of the above, further suppose there are Hilbert spaces $H$, $J$, $K$ continuously embedded into $X$, $Y$, $S$ respectively.
Let $(\psi,\phi,\omega)\in X\times Y\times S$ be the unique solution to
\begin{align*}
	c(\eta,\psi) + b(\eta,\phi) + (T\eta,\omega)_S &= (u-u_h^l,\eta)_H \quad &\forall \eta \in X,
	\\
	b(\psi,\xi) - m(\phi,\xi) &= (w-w_h^l,\xi)_J \quad &\forall \xi \in Y,
	\\
	(T\psi,z)_S &= (\lambda-\lambda_h^l,z)_K \quad &\forall z\in S.
\end{align*}
Assume there are Banach spaces $\hat{X}$, $\hat{Y}$, $\hat{S}$ continuously embedded in $X$, $Y$, $S$ respectively, with $(\psi,\phi,\omega)\in \hat{X}\times\hat{Y}\times\hat{S}$ and there is $\hat{C},\beta>0$ such that
\[
	\inf_{(\eta_h,\xi_h,\chi_h)\in X_h \times Y_h\times S_h}
		\left( \|\psi - \eta_h^l\|_X + \|\phi-\xi_h^l\|_Y + \|\omega - \chi_h^l\|_S \right)
		\leq
		\hat{C}h^\beta(\|\psi\|_{\hat{X}} + \|\phi\|_{\hat{Y}} + \|\omega\|_{\hat{S}}).
\]
Finally assume the regularity result of
\[
	\|\psi\|_{\hat{X}} + \|\phi\|_{\hat{Y}} + \|\omega\|_{\hat{S}}
	\leq
	\tilde{C}(\|u-u_h^l\|_H + \|w-w_h^l\|_J + \|\lambda-\lambda_h^l\|_K).
\]
Then for sufficiently small $h$, there is $C>0$ independent of $h$ such that
\[
	\|u-u_h^l\|_H + \|w-w_h^l\|_J + \|\lambda-\lambda_h^l\|_K
	\leq
	Ch^{\min(\alpha+\beta,k)}
	(\|u\|_{\bar{X}} + \|w\|_{\bar{Y}} + \|\lambda\|_{\bar{S}} + \|f\|_{X^*} + \|g\|_{Y^*} + \|s\|_{S}).
\]
\end{proposition}
\begin{proof}
Let $(\psi,\phi,\omega)$ be as above, then by testing the system with $(u-u_h^l,w-w_h^l,\lambda - \lambda_h^l)$ and adding together, one has,
\begin{align*}
( u-u_h^l,&u-u_h^l )_H + ( w-w_h^l,w-w_h^l)_J + ( \lambda-\lambda^l,\lambda-\lambda^l )_K
\\
=
c&(u-u_h^l,\psi-\eta_h^l) + b(u-u_h^l,\phi-\xi_h^l) + (T(u-u_h^l),\omega - \chi_h^l)_S
\\&+
b(\psi-\eta_h^l,w-w_h^l) - m(\phi-\xi_h^l,w-w_h^l)
+
(T(\psi-\eta_h^l),\lambda-\lambda_h^l)_S
\\&+
\langle f,\eta_h^l\rangle + \langle g,\xi_h^l\rangle + (s,\chi_h^l)_S
\\&-
c(u_h^l,\eta_h^l) - b(u_h^l,\xi_h^l) - (Tu_h^l,\chi_h^l)_S
-
b(\eta_h^l,w_h^l) + m(\xi_h^l,w_h^l)
-
(T\eta_h^l,\lambda_h^l)_S
\\
&-
\langle f_h, \eta_h\rangle - \langle g_h, \xi_h\rangle - (s_h,\chi_h)_{S_h}
\\
&+
c_h(u_h,\eta_h) + b_h(u_h,\xi_h) + (T_h u_h,\chi_h)_{S_h}
+
b_h(\eta_h,w_h) - m_h(\xi_h,w_h)
+
(T_h \eta_h, \lambda_h)_{S_h}.
\end{align*}
we see that the first two lines may be bounded by
\[
	(\|\psi-\eta_h^l\|_X + \|\phi-\xi_h^l\|_Y + \|\omega- \chi_h^l \|_S)(\|u-u_h^l\|_X + \|w-w_h^l\|_Y + \|\lambda-\lambda_h^l\|_S),
\]
and the final four lines may be bounded by the approximation properties 
of the discrete operators.
One then obtains
\begin{align*}
	( u-&u_h^l,u-u_h^l )_H + ( w-w_h^l,w-w_h^l)_J + ( \lambda-\lambda^l,\lambda-\lambda^l )_K
	\\
	\leq&
	C\left[\left(\|\psi-\eta_h^l\|_X + \|\phi-\xi_h^l\|_Y + \|\omega- \chi_h^l \|_S\right) \left(\|u-u_h^l\|_X + \|w-w_h^l\|_Y + \|\lambda-\lambda_h^l\|_S \right) \right.
	\\
	+&\left.h^k \left( \|f\|_{X^*} + \|g \|_{Y^*} + \| s\|_S  \right)\left(\|\psi- \eta_h^l\|_X + \|\phi- \xi_h^l\|_Y + \|\omega-\chi_h^l\|_S + \|\psi\|_X + \| \phi\|_Y + \| \omega \|_S \right) \right].
\end{align*}
By taking infimum over $(\eta_h,\xi_h,\chi_h) \in X_h\times Y_h\times S_h$, one has the result by use of Young's inequality.
\end{proof}

\subsection{Finite element method for the penalty problem}
\begin{problem}\label{discretePenaltyProblem}
Given $\epsilon>0$, find $(u_h^\epsilon,w_h^\epsilon) \in X_h \times Y_h$ solving the problem
\begin{align*}
	c_h(u_h^\epsilon,\eta_h) + b(\eta_h,w_h^\epsilon) + \frac{1}{\epsilon}(T_h u_h^\epsilon, T_h \eta_h)_{S_h} &= \langle f_h, \eta_h\rangle + \frac{1}{\epsilon}(s_h,T_h \eta_h)_{S_h} \, &\forall \eta_h \in X_h,
	\\
	b_h(u_h^\epsilon,\xi_h) - m_h(w_h^\epsilon,\xi_h) &= \langle g_h,\xi_h \rangle \, &\forall \xi_h \in Y_h.
\end{align*}
\end{problem}

We now prove well-posedness of this problem and give error estimates.
For the error estimate, there are two obvious ways to proceed, first of all, one might wish to consider proceeding as though this is a problem independent of the hard constraint problem.
An alternate approach is to use that the hard constraint problem is well approximated by the penalty problem and show a similar bound for the discrete problems, then use the error estimates for the hard constraint problem.
We start by showing the existence of a solution to Problem \ref{discretePenaltyProblem}.
We suppose that Assumptions \ref{ass:AbstractCoercivity}, \ref{kernelInfSup}, \ref{InfSupSOSCoercivity}, \ref{ass:discreteInfSupB} and \ref{ass:DiscreteLagrange} hold true for the remainder of this subsection.
\begin{theorem}\label{ThmDiscretePenaltyProblem}
	For sufficiently small $h$ and sufficiently small $\epsilon$, there exists a unique solution to Problem \ref{discretePenaltyProblem}.
\end{theorem}
\begin{proof}
Existence and uniqueness follows from the homogeneous case $f_h= g_h =s_h =0$ as the system is linear and finite dimensional.
	In this homogenous case we have that, by testing the first equation of the system with $u_h^\epsilon$ and the second equation of the system with $w_h^\epsilon$ and taking differences,
	\begin{equation}\label{discreteHomogeneousDifference}
		c_h(u_h^\epsilon,u_h^\epsilon) + \frac{1}{\epsilon}(T_h u_h^\epsilon, T_h u_h^\epsilon)_{S_h} + m_h(w_h^\epsilon,w_h^\epsilon) = 0.
	\end{equation}
It follows from Lemma \ref{lemma:discreteCoercivity} that
\begin{align*}
	C\|(w_h^\epsilon)^l\|^2_L
	\leq&
	c(G_h^l m((w^\epsilon_h)^l,\cdot),G_h^l m((w^\epsilon_h)^l,\cdot))
	\\
	&+ \frac{1}{\epsilon_0}(T(G_h^l m((w^\epsilon_h)^l,\cdot)),T(G_h^l m((w^\epsilon_h)^l,\cdot)))_{S}+ m((w_h^\epsilon)^l,(w_h^\epsilon)^l)
	\\
	\leq&
	c((u_h^\epsilon)^l,(u_h^\epsilon)^l) +m((w_h^\epsilon)^l,(w_h^\epsilon)^l)
-
	c_h(u_h^\epsilon,u_h^\epsilon) -m_h(w_h^\epsilon,w_h^\epsilon)
	\\
	&+
	c(G_h^l m((w^\epsilon_h)^l,\cdot),G_h^l m((w^\epsilon_h)^l,\cdot))
	-
	c((u_h^\epsilon)^l,(u_h^\epsilon)^l)
	\\
	&+ \frac{1}{\epsilon_0}(T(G_h^l m((w^\epsilon_h)^l,\cdot)),T(G_h^l m((w^\epsilon_h)^l,\cdot)))_{S} - \frac{1}{\epsilon_0}(T_h u_h^\epsilon,T_hu_h^\epsilon)_{S_h}
	\\
	\leq&
	\tilde{C}h^k \|(w_h^\epsilon)^l\|_L^2,
\end{align*}
as in the proof of Theorem \ref{ThmDiscreteLagrangeProblem}, where we have inequality, rather than equality for
\[
	c_h(u_h^\epsilon,u_h^\epsilon) + m_h(w_h^\epsilon,w_h^\epsilon) + \frac{1}{\epsilon_0} (T_h u_h^\epsilon,T_hu_h^\epsilon)_{S_h} \leq 0
\]
Hence for sufficiently small $h$, it holds that $w_h^l = u_h^l = 0$.
Thus there is a unique solution.
\end{proof}

We now wish to show a discrete version of Proposition \ref{Prop:FirstAbstractEstimate} so that we may use the approximation theory from Theorem \ref{ThmDiscreteLagrangeProblem} to obtain uniform estimates on the solution to Problem \ref{discretePenaltyProblem}.
\begin{lemma}\label{discreteConvergenceInEps}
Let $(u_h^\epsilon,w_h^\epsilon)$ solve Problem \ref{discretePenaltyProblem} and let $(u_h,w_h,\lambda_h)$ solve Problem \ref{discreteLagrangeProblem}.
Then for sufficiently small $\epsilon>0$, there is $C>0$ independent of $\epsilon$ and $h$ such that,
\[
	\|u_h^l - (u_h^\epsilon)^l\|_X
	+	\|w_h^l - (w_h^\epsilon)^l\|_Y
	+	\left\lVert \lambda_h^l - {\epsilon}^{-1}T\left((u_h^\epsilon)^l- u_h^l\right)\right\rVert_S
	\leq C \|\lambda_h^l\|_S \sqrt{\epsilon}.
\]
\end{lemma}
\begin{proof}
As previously shown in the proof of Theorem \ref{ThmDiscreteLagrangeProblem},
\[
	\| G_h^lm((w_h^\epsilon)^l -w_h^l,\cdot) - \left( (u_h^\epsilon)^l - u_h^l\right) \|_X
	\leq
	Ch^k (\|u_h^l - (u_h^\epsilon)^l\|_X + \|w_h^l - (w_h^\epsilon)^l\|_L),
\]
where we also know for $h$ sufficiently small,
\[
	\|u_h^l - (u_h^\epsilon)^l\|_X \leq C \|w_h^l - (w_h^\epsilon)^l\|_L.
\]
Adding $({\epsilon}^{-1} - \epsilon_0^{-1})\| T((u_h^\epsilon)^l - u_h^l) \|_S^2$, to the statement of Lemma \ref{lemma:discreteCoercivity}, with the above, yields,
\begin{align*}
	C\|w_h^l - (w_h^\epsilon)^l \|_L^2 +& ({\epsilon}^{-1} - \epsilon_0^{-1})\| T((u_h^\epsilon)^l - u_h^l) \|_S^2
	\\
	\leq&
	c( G_h^lm((w_h^\epsilon)^l - w_h^l, \cdot), G_h^lm((w_h^\epsilon)^l - w_h^l, \cdot) ) + m((w_h^\epsilon)^l - w_h^l,(w_h^\epsilon)^l - w_h^l)
	\\
	&+ {\epsilon}^{-1}(T((u_h^\epsilon)^l - u_h^l),T((u_h^\epsilon)^l - u_h^l))_{S}
	\\
	\leq&
	c((u_h^\epsilon)^l - u_h^l,(u_h^\epsilon)^l - u_h^l) + m((w_h^\epsilon)^l - w_h^l,(w_h^\epsilon)^l - w_h^l)
	\\
	&+ {\epsilon}^{-1}(T((u_h^\epsilon)^l - u_h^l),T((u_h^\epsilon)^l - u_h^l))_{S} +Ch^k\|(w_h^\epsilon)^l - w_h^l\|_L^2.
\end{align*}
For sufficiently small $h$, we may smuggle the $h^k$ terms into the left hand side, giving,
\begin{align*}
	C\|w_h^l - (w_h^\epsilon)^l \|_L^2 &+ ({\epsilon}^{-1} - \epsilon_0^{-1})\| T((u_h^\epsilon)^l - u_h^l) \|_S^2
	\\
	\leq&
	c((u_h^\epsilon)^l - u_h^l,(u_h^\epsilon)^l - u_h^l) + m((w_h^\epsilon)^l - w_h^l,(w_h^\epsilon)^l - w_h^l)
	\\
	&+
	{\epsilon}^{-1}(T((u_h^\epsilon)^l - u_h^l),T((u_h^\epsilon)^l - u_h^l))_{S}
	\\
	\leq&
	c_h(u_h^\epsilon - u_h,u_h^\epsilon - u_h) + m_h(w_h^\epsilon- w_h,w_h^\epsilon - w_h)
	+
	{\epsilon}^{-1}(T_h(u_h^\epsilon - u_h),T_h(u_h^\epsilon - u_h))_{S_h}
	\\
	&+ c((u_h^\epsilon)^l - u_h^l,(u_h^\epsilon)^l - u_h^l) -c_h(u_h^\epsilon - u_h,u_h^\epsilon - u_h) 
	\\
	&+ m((w_h^\epsilon)^l - w_h^l,(w_h^\epsilon)^l - w_h^l) - m_h(w_h^\epsilon- w_h,w_h^\epsilon - w_h)
	\\
	&+
	{\epsilon}^{-1}(T((u_h^\epsilon)^l - u_h^l),T((u_h^\epsilon)^l - u_h^l))_{S} - {\epsilon}^{-1}(T_h(u_h^\epsilon - u_h),T_h(u_h^\epsilon - u_h))_{S_h},
\end{align*}
where the final three lines may be bounded by $Ch^k \left(\|w_h^l - (w_h^\epsilon)^l \|_L^2 +{\epsilon}^{-1}\| T((u_h^\epsilon)^l - u_h^l) \|_S^2 \right)$, which follows from the approximations in Assumption \ref{ass:AbstractApproximationProperties}.
Thus for $h$ sufficiently small, one has
\begin{align*}
	C\left( \|w_h^l - (w_h^\epsilon)^l \|_L^2 + ({\epsilon}^{-1} - \epsilon_0^{-1})\| T((u_h^\epsilon)^l - u_h^l) \|_S^2\right)
	\leq&
	c_h(u_h^\epsilon - u_h,u_h^\epsilon - u_h) + m_h(w_h^\epsilon- w_h,w_h^\epsilon - w_h)
	\\
	&+
	{\epsilon}^{-1}(T_h(u_h^\epsilon - u_h),T_h(u_h^\epsilon - u_h))_{S_h}
	\\
	=&
	(\lambda_h, T_h(u_h^\epsilon - u_h) )_{S_h}
	\leq C_2\|\lambda_h^l\|_S \|T(u_h^\epsilon - u_h)\|_S
	\\
	\leq&
	C_2^2 \epsilon\rho \| \lambda_h^l\|_S^2 + \frac{1}{\epsilon\rho}\|T(u_h^\epsilon - u_h)\|_S^2,
\end{align*}
where the equality follows from the discrete equations  and the final line is Young's inequality.
Thus choosing $\rho$ sufficiently big (independent of $\epsilon$) gives
\[
	\|w_h^l - (w_h^\epsilon)^l \|_L^2 + {\epsilon}^{-1}\| T((u_h^\epsilon)^l - u_h^l) \|_S^2
	\leq
	C\epsilon\|\lambda_h^l\|_S^2.
\]
We now make use of Assumption \ref{ass:DiscreteLagrange},
\begin{align*}
	\tilde{\kappa} (\|w_h^l - &(w_h^\epsilon)^l\|_Y + \|\lambda_h^l - {\epsilon}^{-1}T((u_h^\epsilon)^l - u_h^l) \|_S )
	\\&\leq
	\sup_{(\eta_h, \xi_h) \in X_h\times Y_h}
	\frac{b(\eta_h^l,w_h^l - (w_h^\epsilon)^l) +(T\eta_h^l, \lambda_h^l - {\epsilon}^{-1}T((u_h^\epsilon)^l - u_h^l))_S + m(w_h^l - (w_h^\epsilon)^l,\xi_h^l) }{\|\eta_h^l\|_X + \|\xi_h^l\|_Y}.
\end{align*}
It is clear in the above that the $m$ term is bounded as we would like.
We then have
\begin{align*}
	b(\eta_h^l,w_h^l - &(w_h^\epsilon)^l) + (T\eta_h^l, \lambda_h^l - {\epsilon}^{-1}T((u_h^\epsilon)^l - u_h^l))_S
	\\
	=&
	b_h(\eta_h, w_h^\epsilon -w_h) + (T_h\eta_h, \lambda_h - {\epsilon}^{-1}T_h (u_h^\epsilon - u_h))_{S_h}
	+
	b(\eta_h^l,w_h^l - (w_h^\epsilon)^l) - b_h(\eta_h, w_h^\epsilon -w_h) 
	\\
	&+
	(T\eta_h^l, \lambda_h^l - {\epsilon}^{-1}T((u_h^\epsilon)^l - u_h^l))_S - (T_h\eta_h, \lambda_h - {\epsilon}^{-1}T_h ( u_h^\epsilon - u_h))_{S_h}
	\\
	=&
	c_h(u_h-u_h^\epsilon,\eta_h)
	+
	b(\eta_h^l,w_h^l - (w_h^\epsilon)^l) - b_h(\eta_h, w_h^\epsilon -w_h) 
	\\
	&+
	(T\eta_h^l, \lambda_h^l - {\epsilon}^{-1}T((u_h^\epsilon)^l - u_h^l))_S - (T_h\eta_h, \lambda_h - {\epsilon}^{-1}T_h( u_h^\epsilon - u_h))_{S_h}
	\\
	\leq&
	c_h(u_h-u_h^\epsilon,\eta_h)
	+
	Ch^k\|\eta_h^l\|_X \left( \|w_h^l - (w_h^\epsilon)^l\|_Y + \|\lambda_h^l - {\epsilon}^{-1}T((u_h^\epsilon)^l - u_h^l) \|_S\right).
\end{align*}
Thus, for $h$ sufficiently small
\[
	\|w_h^l - (w_h^\epsilon)^l\|_Y + \|\lambda_h^l - {\epsilon}^{-1}T((u_h^\epsilon)^l - u_h^l) \|_S
	\leq
	C\|u_h^l-(u_h^\epsilon)^l\|_X,
\]
which completes the result.
\end{proof}
This results in the following Theorem.
\begin{theorem}\label{thm:UseApproximationProperties}
	Let $(u^\epsilon_h,w_h^\epsilon)$ be the solution to Problem \ref{discretePenaltyProblem}, let $(u^\epsilon,w^\epsilon)$ be the solution to Problem \ref{penaltyProblem} and let $(u,w,\lambda)$ be the solution to problem \ref{lagrangeProblem}.
	Then there is $C>0$ independent of $h$ and $\epsilon$ such that
	\begin{align*}
		\|u^\epsilon-(u_h^\epsilon)^l\|_X + \|w^\epsilon-(w_h^\epsilon)^l\|_Y
		\leq
		C\inf_{(\eta_h,\xi_h)\in X_h \times Y_h}&
		\left( \|u - \eta_h^l\|_X + \|w-\xi_h^l\|_Y \right)
		\\
		&+
		C(h^k+\sqrt{\epsilon}) \left( \|f\|_{X^*} + \|g\|_{Y^*} + \|s\|_S\right).
	\end{align*}
\end{theorem}
\begin{proof}
	We start by considering $u^\epsilon-(u_h^\epsilon)^l = (u^\epsilon - u) + (u - u_h^l) +(u_h^l-(u_h^\epsilon)^l)$ and similarly for $w$ terms.
	From Proposition \ref{Prop:FirstAbstractEstimate} and Lemma \ref{discreteConvergenceInEps}, it holds
	\[
		\|u^\epsilon-(u_h^\epsilon)^l\|_X + \|w^\epsilon-(w_h^\epsilon)^l\|_Y
		\leq
		C \sqrt{\epsilon}(\|\lambda\|_S+\|\lambda_h^l\|_S) + \|u - u_h^l\|_X + \|w-w_h^l\|_Y,
	\]
	which combined with Theorem \ref{ThmDiscreteLagrangeProblem} gives the result.
\end{proof}

If $(u^\epsilon,w^\epsilon)$ were to be sufficiently more regular than $(u,w,\lambda)$, one may wish to use this extra regularity to pay for the $\epsilon$ cost and obtain higher order convergence than would be attained from Proposition \ref{prop:DiscreteLagrangeEstimates}.

\section{Surface calclus and finite elements}
We recall some definitions and results from surface PDE and surface finite element methods, for full details, the reader is referred to \cite{DziEll13}.
\subsection{Surface calculus}
Let $\Gamma$ be a closed $C^k$ hypersurface in $\R^3$ where $k$ is as large as needed but at most $4$ and at least $2$.
There is a bounded domain $U\subset\R^3$ such that $\partial U = \Gamma$.
The unit normal $\nu$ to $\Gamma$ that points out of this domain $U$ is called the outwards unit normal.
We write $P_\Gamma := I - \nu \otimes \nu$ on $\Gamma$ to be, at each point on $\Gamma$, the projection onto the tangent space of $\Gamma$ at that particular point, where we are writing $I$ to be the $3\times 3$ identity matrix.
For a differentiable function $f$ on $\Gamma$ we define the surface gradient by
\[
	\nabla_\Gamma f := P_\Gamma \nabla \bar{f},
\]
where $\bar{f}$ is a differentiable extension of $f$ to an open neighbourhood of $\Gamma$.
Here $\nabla$ denotes the standard gradient in $\R^3$.
This given definition of the surface gradient depends only on the values of $f$ on $\Gamma$, this is shown in \cite[Lemma 2.4]{DziEll13}.
For a twice differentiable function, the Laplace-Beltrami operator is defined by,
\[
	\Delta_\Gamma f := \nabla_\Gamma \cdot \nabla_\Gamma f.
\]
For convenience we use the following inner product
\[
	(u,v)_{H^2(\Gamma)} := \int_\Gamma \Delta_\Gamma u \Delta_\Gamma v + \nabla_\Gamma u \cdot \nabla_\Gamma u + uv.
\]
Note that this is not the standard inner product on $H^2(\Gamma)$ which contains mixed derivatives.
On closed surfaces however, this is seen to be an equivalent norm, see \cite{DziEll13} for details.
\subsection{Surface finite elements}
We assume that the surface $\Gamma$ is approximated by a polyhedral hypersurface
\[
	\Gamma_h = \bigcup_{T\in \mathcal{T}_h} K,
\]
where $\mathcal{T}_h$ is a set of two-dimensional simplices in $\R^3$ which form an admissible triangulation.
For $K\in \mathcal{T}_h$ the diameter of $K$ is $h(K)$ and the radius of the largest ($2$-dimensional) ball contained in $T$ is $\rho(K)$.
Set $h:= \max_{K\in\mathcal{T}_h} h(T)$ and assume the ratio between $h$ and $\rho(K)$ is bounded independently of $h$.
We assume that $\Gamma_h$ is contained within a narrow strip $\mathcal{N}_\delta$ of width $\delta>0$ around $\Gamma$ on which the decomposition
\[
	x= p + d(x)\nu(p), ~~ p\in\Gamma
\]
is unique for all $x \in \mathcal{N}_\delta$.
Here, $d(x)$ denotes the oriented distance function to $\Gamma$, see \cite[Section 2.3]{DziEll13}.
This defines a map $x\mapsto p(x)$ from $\mathcal{N}_\delta$ to $\Gamma$.
We assume that the restriction $p|_{\Gamma_h}$ of this map on the polyhedral surface is a bijection between $\Gamma_h$ and $\Gamma$.
In addition the vertices of $K\in \mathcal{T}_h$ should lie on $\Gamma$.

The piecewise affine Lagrange finite element space on $\Gamma_h$ is
\[
	\mathcal{S}_h := \{\chi \in C(\Gamma_h) \,:\, \chi|_K \in P^1(K)\, \forall K \in \mathcal{T}_h\},
\]
where $P^1(K)$ is the set of polynomials of degree 1 or less on $K$.
The Lagrange basis functions $\phi_i$ of this space are uniquely determined by their values at the so-called Lagrange nodes $q_j$, that is $\phi_i(q_j) = \delta_{ij}$.
The associated Lagrange interpolation for a continuous function $f$ on $\Gamma_h$ is defined by
\[
	I_h f:= \sum_i f(q_i) \phi_i.
\]
We now introduce the lifted discrete spaces.
We use the standard lift operator as constructed in \cite[Section 4.1]{DziEll13}.
The lift $f^l$ of a continuous function $f \colon \Gamma_h \to \R$ onto $\Gamma$ is defined by
\[
	f^l(x) := (f\circ p|_{\Gamma_h}^{-1})(x)
\]
for all $x \in \Gamma$.
The inverse map $g^{-l}$ for a continuous function $g \colon \Gamma \to R$ onto $\Gamma_h$ is given by $g^{-l} := g\circ p$.
The lifted finite element space is
\[
	\mathcal{S}_h^l := \{ \chi^l \,|\, \chi \in S_h\}.
\]
With the lifted Lagrange interpolation $I_h^l \colon C(\Gamma)\to \mathcal{S}_h^l$ given by $I_h^l(f) := (I_h f^{-l})^l$.
The lifted discrete spaces satisfy the conditions in Assumption \ref{ass:discreteInfSupB} when $X_h^l = Y_h^l := \mathcal{S}_h^l$, or specifically, for a sequence of triangulated surfaces $(\Gamma_{h_h})_{n\in\mathbb{N}}$ with $h_n \to 0$ as $n \to \infty$ we have $X_n:= X_{h_n}^l = \mathcal{S}_h^l$ and $Y_n:= Y_{h_n}^l = \mathcal{S}_h^l$.

\section{PDE examples}\label{Sec:SurfaceApplication}

\subsection{A near-spherical biomembrane}\label{subsec:NearSphericalBiomembrane}
We now discuss the second order splitting associated to a fourth order linear PDE which arises in the modelling of biomembranes on near spherical domains, in particular, we now set $\Gamma = \mathbb S^2(0,R)$ for some fixed $R>0$.
The model is based on the deformations of the membrane due to small external forcing.
A full derivation may be found in \cite{EllFriHob17}, see also \cite{EllHatHer19}.
Here we are considering that the only contributing deformation is due to  $N$ point constraints or point penalties.

We now fix $N\in \mathbb{N}$ and distinct $X_i \in \Gamma$ for $i=1,...,N$, with $S:= \mathbb R^N$.
We set $\kappa >0$ and $\sigma \geq 0$.
First we define the energies and then  give the hard constraint (Lagrange) problem and a soft constraint (penalty) problem.

\subsubsection{Fourth order formulation}
\begin{definition}\label{def:4OrderSurface}
We define $a\colon H^2(\Gamma) \times H^2(\Gamma) \to \R$ by
\[
	a(u,v) := \int_\Gamma \kappa \Delta_\Gamma u \Delta_\Gamma v +\left(\sigma -\frac{2\kappa}{R^2} \right) \nabla_\Gamma u \cdot \nabla_\Gamma v - \frac{2\sigma}{R^2}uv,
\]
 $T\colon C(\Gamma) \to \R^N$  by
\[
	Tu = (u(X_i))_{i=1}^N,
\]
and for any $\epsilon>0$, $a_\epsilon\colon H^2(\Gamma) \times H^2(\Gamma) \to \R$ by
\[
	a_\epsilon (u,v) := a(u,v) +  \frac{1}{\epsilon}(Tu , Tv)_{\R^N}.
\]
\end{definition}
We notice that over $H^2(\Gamma)$ that neither $a_\epsilon$ nor $a$ are necessarily coercive, however, in \cite{EllFriHob17} it is seen that they are coercive over $\{1,\nu_1,\nu_2,\nu_3\}^\perp$, where $\nu= \frac{x}{R}$ is the unit normal to $\Gamma$ and $\perp$ is meant in the sense of $H^2(\Gamma)$.
This is a consequence of the fact that $1,\nu_1,\nu_2,\nu_3$ are eigenfunction of $-\Delta_\Gamma$.
Furthermore, under suitable conditions on the location of the points $\{X_i\}_{i=1}^N$ we show in the following proposition that  both $a$ and $a_\epsilon$ are coercive over $\{1\}^\perp$.
We use the following notation, $U:= \{ v \in H^2(\Gamma) : \int_\Gamma v = 0\}$ and $U_0:= \{ v \in U : Tv = 0\}$.
\begin{proposition}\label{prop:SurfaceCoercivity}
	Suppose $N\geq4$ and $\{X_i\}_{i=1}^N$ do not lie in the same plane.
	Then there is $\epsilon_0>0$ and $C>0$ such that for any $\epsilon<\epsilon_0$
	\begin{align*}
		a_\epsilon(\eta,\eta) \geq& C \|\eta\|_{2,2}^2 ~~ \forall \eta \in U,
		\\
		a(\eta,\eta) \geq& C\|\eta\|_{2,2}^2 ~~ \forall \eta \in U_0.
	\end{align*}
\end{proposition}
\begin{proof}
	We notice that for $u\in U_0$, it holds that $a_\epsilon(u,u) = a(u,u)$, thus we need only show the first result.
	In \cite[Proposition 4.4.2]{Hob16}, it is shown that $a_\epsilon$ is coercive over $\big( \mbox{Sp}\{ 1,\nu_1,\nu_2,\nu_3\} \cap \Ker(T)\big)^{\perp} \cap U$.
	Thus it is sufficient to show that  $\left( \mbox{Sp} \{ 1,\nu_1,\nu_2,\nu_3\} \cap \Ker(T)\right) = \{0\}$.
	Let $v \in {\rm Sp}\{1,\nu_1,\nu_2,\nu_3\}\cap \Ker(T)$, one has that
	\[
		v = \alpha_0 + \sum_{j=1}^3 \alpha_j \nu_j.
	\]
	By making note that $\nu_j(x) = \frac{x_{j}}{R}$ we see that $v$ is an affine function.
	The condition $v \in \Ker(T)$ gives that the $X_i$ are in the zero level set of the affine function $v$.
	Thus the points must lie in the same plane or $v\equiv 0$.
\end{proof}
From now on, we assume that $\{X_i\}_{i=1}^N$ do not lie in the same plane. Notice that for $f = 0$, the following problems are the membrane problems in \cite{EllFriHob17,EllHatHer19}.
\begin{problem}\label{ProblemLagrangeApplication}
Given $f \in (H^2(\Gamma))^*$, find $u \in U$ minimising $\frac{1}{2} a(u,u) - \langle f, u\rangle$ subject to $u(X_i) = Z_i$ for $i=1,...,N$.
This has the variational formulation of finding $u \in U$ such that $u(X_i)=Z_i$ for $i=1,...,N$ and
\[
	a(u,\eta) = \langle f, \eta\rangle ~~ \forall \eta \in U_0.
\]
\end{problem}
\begin{problem}\label{ProblemPenaltyApplication}
Given $f \in (H^2(\Gamma))^*$, find $u^\epsilon \in U$ minimising $\frac{1}{2} a(u^\epsilon,u^\epsilon) + \frac{1}{2\epsilon}|Tu^\epsilon - Z|^2 - \langle f, u^\epsilon \rangle$.
This has the variational formulation of finding $u^\epsilon \in U$ such that
\[
	a_\epsilon(u^\epsilon,\eta) = \frac{1}{\epsilon} (Z,T\eta)_{\R^N} + \langle f, \eta \rangle ~~ \forall \eta \in U.
\]
\end{problem}
\begin{theorem}
	There are unique solutions to both Problems \ref{ProblemLagrangeApplication} and \ref{ProblemPenaltyApplication}.
\end{theorem}
\begin{proof}
This is an application of Lax-Milgram with the coercivity of the bilinear forms shown in Proposition \ref{prop:SurfaceCoercivity}.
\end{proof}

In order to write down the PDE associated to these problems we need to extend the variational formulation to be posed over the whole of $H^2(\Gamma)$.  Standard arguments yield that the solution of Problem \ref{ProblemLagrangeApplication} solves
\[
	a(u,\eta) + (\lambda,T\eta)_{\R^N} + \bar{p}\int_\Gamma \eta = \langle f,\eta\rangle ~~ \forall \eta \in H^2(\Gamma).
\]
and
\[
	a(u,\eta) + (\lambda, T \eta)_{\R^N} = \langle f,\eta\rangle~~ \forall \eta \in U.
\]
Thus by considering smooth test functions, this gives the distributional PDE
	\begin{align*}
		\kappa \Delta_\Gamma^2 u- \left( \sigma - \frac{2\kappa}{R^2} \right)\Delta_\Gamma u - \frac{2\sigma}{R^2} u + \bar{p} + \sum_{i=1}^N \lambda_i \delta_{X_i}=& f~\mbox{in}~ \Gamma,
		\\
		\int_\Gamma u = 0, ~ u(X_i)=&Z_i~\mbox{for}~i=1,...,N.
	\end{align*}
It is useful to to note the following.
\begin{proposition}\label{prop:regularityOfSol}
The unique solution of Problem \ref{ProblemLagrangeApplication}, $u$, satisfies $u\in  W^{3,p}(\Gamma),\, p\in (1,2)$ and is given by $$u = u_f + \sum_{k=1} Z_k \phi_k.$$ 
where
\begin{itemize}
\item	

  For $k=1,...,N$,  the unique  $\phi_k\in U$ $\phi_k(X_j) = \delta_{jk}$ for $j=1,..,N$ and
\[
	a(\phi_k,\eta) = 0 ~~\forall \eta \in U_0
\]
and the unique  $u_f \in U_0$ such that
\[
	a(u_f,\eta) = \langle f, \eta \rangle ~~ \forall \eta \in U_0.
\]

\item
For each $k=1,...,N$,
\[
	\lambda_k = \langle f,\phi_k\rangle - a(u_f,\phi_k) - \sum_{j=1}^N Z_j a(\phi_j,\phi_k)
\]

and $$\bar{p} = \langle f,\phi_0 \rangle - a(u_f,\phi_0)- \sum_{k=1}^N Z_k a(\phi_k,\phi_0)$$
where $\phi_0$ uniquely satisfies  $\phi_0(X_j) = 0$ for $j=1,...,N$, $\int_\Gamma \phi_0 =1$,
\[
	a(\phi_0,\eta) = 0 ~~ \forall \eta \in U_0.
\]
\end{itemize}

\end{proposition}
\begin{proof} The formulae for the solution are easily verified.
	Since $\bar p,\lambda$ are bounded in terms of the data, regularity for this fourth order equation on the sphere yields $u\in W^{3,p}(\Gamma), p\in (1,2)$, following the arguments for fourth order equations in the flat case \cite{Cas85,Hob16}.
\end{proof}

\begin{remark}
	For the variational problem with penalty, the variational formulation over the whole domain follows similarly, yielding the distributional PDE,
\begin{align*}
	\kappa \Delta_\Gamma^2 u^\epsilon- \left( \sigma - \frac{2\kappa}{R^2} \right)\Delta_\Gamma u^\epsilon - \frac{2\sigma}{R^2} u + \bar{p}^\epsilon + \frac{1}{\epsilon}\sum_{i=1}^N u^\epsilon\delta_{X_i}&= f + \frac{1}{\epsilon}\sum_{i=1}^n  Z_i\delta_{X_i}~\mbox{on}~ \Gamma,
		\\
	\int_\Gamma u^\epsilon &= 0.
\end{align*}
\end{remark}

\subsubsection{Second order splitting} 
For certain boundary value problems, splitting a fourth order equation into two second order equations is a natural approach,  c.f. \cite{EllFreMil89}. Here it is convenient to use an auxiliary variable $w=-\Delta_\Gamma u+u$ leading to the following coupled system holding on $\Gamma$

\begin{equation}\label{secondordersplitting}
	\begin{split}
		-\Delta_\Gamma w+w-\left(\frac{\sigma}{\kappa}-\frac{2}{R^2}-2\right)\Delta_\Gamma u-\left(\frac{2\sigma}{\kappa R^2}+1\right)u &=\frac{1}{\kappa}\left(f-\bar p+\Sigma_{i=1}^N\lambda_i\delta_{X_i}\right),\\
		-\Delta_\Gamma u+u &=w.
		\end{split}
	\end{equation}
Taking $u\in W^{3,p}(\Gamma)$ as the solution of Problem \ref{ProblemLagrangeApplication}, because of the Dirac measures on the right hand side we  pose the first equation weakly in the dual of $W^{1,q}(\Gamma)$ with $q\in (2,\infty)$ and $u\in W^{1,q}(\Gamma)$ with $\frac{1}{p}+\frac{1}{q}=1$. The second equation is posed in the dual of $W^{1,p}(\Gamma)$. 
It is clear by testing \eqref{secondordersplitting} by $\eta \in X$ and $\xi \in Y$ that the PDE system may be posed as in \eqref{eq:FirstStatement} using  Definition \ref{def:applicationBilinearForms}.

\begin{definition}\label{def:applicationBilinearForms}

Let $\infty>q>2>p>1$, then we define the spaces
\begin{align*}
	X=&\left\{\eta \in W^{1,q}(\Gamma) \,:\, \int_\Gamma\eta =0 \right\},
	\\
	Y=&\left\{\xi \in W^{1,p}(\Gamma) \,:\, \int_\Gamma\xi=0\right\},
\end{align*}
$L=L^2(\Gamma)$ and $S = \R^N$, with the bilinear forms,
\begin{align*}
	c\colon& W^{1,q}(\Gamma) \times W^{1,q}(\Gamma) \to \R,
	\\
	b\colon& W^{1,q}(\Gamma) \times W^{1,p}(\Gamma) \to \R,
	\\
	m\colon& L^{2}(\Gamma) \times L^{2}(\Gamma) \to \R,
\end{align*}
given by
\begin{align*}
	c(u,\eta)
	&=
	\int_\Gamma \left(\frac{\sigma}{\kappa} - 2 - \frac{2}{R^2}\right) \nabla_\Gamma u \cdot \nabla_\Gamma \eta
	-
	\left(1+\frac{2\sigma}{\kappa R^2}\right)u\eta,
	\\
	b(\eta,\xi)
	&=
	\int_\Gamma \nabla_\Gamma \eta \cdot \nabla_\Gamma \xi + \eta\xi,
	\\
	m(\eta,\xi) &= \int_\Gamma \eta\xi
\end{align*}
and the linear operator $T:X\rightarrow S$ by
\[T\eta:=(\eta(X_1),\eta(X_2),...\eta(X_N)),~\eta\in X.\]
\end{definition}

\subsubsection{Verification of Assumptions \ref{ass:AbstractCoercivity}, \ref{InfSupSOSCoercivity}, \ref{ass:discreteInfSupB} and \ref{kernelInfSup} }
We begin with the following three results which are shown in \cite[Section 5]{EllFriHob19}.

\begin{lemma}\label{lem:surfInfSup}
Let $1<p\leq 2 \leq q < \infty$ with $p,q$ conjugate.
There is $\beta,\gamma >0$ such that
\[
	\beta\|\eta\|_{1,q} \leq \sup_{\xi \in W^{1,p}(\Gamma)} \frac{b(\eta,\xi)}{\|\xi\|_{1,p}} \quad \forall \eta \in W^{1,q}(\Gamma)
	\quad \text{and} \quad
	\gamma\|\xi\|_{1,p} \leq \sup_{\eta \in W^{1,q}(\Gamma)} \frac{b(\eta,\xi)}{\|\eta\|_{1,q}} \quad \forall \xi \in W^{1,p}(\Gamma).
\]
\end{lemma}
\begin{lemma}\label{lem:surfRitz}
Let $1<r\leq \infty$.
Then there is a bounded (independently of $h$) linear map\newline $\Pi_h\colon W^{1,r}(\Gamma) \to \mathcal{S}_h^l$ given by
\[
	b(\Pi_h \phi,v_h^l) = b(\phi,v_h^l) \quad \forall v_h^l \in \mathcal{S}_h^l.
\]
It also holds that
\[
	\sup_{\psi \in W^{1,r}(\Gamma)} \frac{\|\psi - \Pi_h\psi\|_{0,2}}{\|\psi\|_{1,r}} \to 0 \quad \text{as} \quad h\to 0.
\]
\end{lemma}
\begin{lemma}\label{lem:surfDiscInfSup}
Let $1<p\leq 2 \leq q < \infty$ with $p,q$ conjugate.
There is $\tilde{\beta},\tilde{\gamma} >0$ such that
\[
	\tilde{\beta}\|\eta_h^l\|_{1,q} \leq \sup_{\xi_h^l \in \mathcal{S}_h^l} \frac{b(\eta_h^l,\xi_h^l)}{\|\xi_h^l\|_{1,p}}\quad \forall \eta_h^l \in \mathcal{S}_h^l
	\quad \text{and} \quad
	\tilde{\gamma}\|\xi_h^l\|_{1,p} \leq \sup_{\eta_h^l \in \mathcal{S}_h^l} \frac{b(\eta_h^l,\xi_h^l)}{\|\eta_h^l\|_{1,q}}\quad \forall \xi_h^l \in \mathcal{S}_h^l.
\]
\end{lemma}
In order to prove well-posedness of the problem with penalty, we are left to show Assumptions \ref{InfSupSOSCoercivity}, \ref{ass:discreteInfSupB} and \ref{ass:AbstractCoercivity} in the appropriate spaces.
We make use of Fortin's criteria with the projection to mean-value-free functions and Lemmas \ref{lem:surfInfSup} and \ref{lem:surfDiscInfSup}.
The following Proposition gives Assumptions \ref{InfSupSOSCoercivity} and \ref{ass:discreteInfSupB}.

\begin{proposition}\label{lem:surfaceMeanValueWork}
Assumptions \ref{InfSupSOSCoercivity} and \ref{ass:discreteInfSupB} hold true.
\end{proposition}
\begin{proof}
	This follows from an application of Fortin's criterion \cite[Lemma 4.19]{ErnGue13} to Lemmas \ref{lem:surfInfSup} and \ref{lem:surfDiscInfSup} with the projection map $\bar{P}\colon u \mapsto u - \frac{1}{|\Gamma|}\int_\Gamma u$.
\end{proof}

The final condition we need to check is the coercivity like condition, Assumption \ref{ass:AbstractCoercivity}.
\begin{lemma}\label{lem:surfSoSCoercivity}
	Let $\epsilon>0$ sufficiently small and assume $(u,w) \in X \times Y$ satisfy
	\[
		b(u,\xi) = m(w,\xi)~~ \forall \xi \in Y.
	\]
	Then there is $C>0$ such that
	\[
		C\|w\|_{0,2}^2 \leq c(u,u) + m(w,w) + \frac{1}{\kappa\epsilon}(Tu,Tu)_{\R^N}.
	\]
\end{lemma}
\begin{proof}
	The condition $b(u,\xi) = m(w,\xi)~~ \forall \xi \in Y$ with $(u,w) \in X\times Y$ implies that $b(u,\xi) = m(w,\xi)~~ \forall \xi \in W^{1,p}(\Gamma)$.
	Elliptic regularity gives $-\Delta_\Gamma u \in W^{1,p}(\Gamma)$ and $-\Delta_\Gamma u + u = w$.
	Making use of this relation in $c(u,u)+m(w,w)$ gives
	\[
		c(u,u) + m(w,w) + \frac{1}{\kappa\epsilon}(Tu,Tu)_{\R^N}
		=
		\int_\Gamma \left( \left(\Delta_\Gamma u\right)^2 + \left(\frac{\sigma}{\kappa} - \frac{2}{R^2}\right)|\nabla_\Gamma u|^2 -\frac{2\sigma}{R^2}u^2 \right)+ \frac{1}{\kappa\epsilon}(Tu,Tu)_{\R^N},
	\]
	which is bounded below by $C\|u\|_{2,2}^2$, as shown in Proposition \ref{prop:SurfaceCoercivity}.
	Elliptic regularity applied to the condition $b(u,\xi) = m(w,\xi)~ \forall \xi \in W^{1,p}(\Gamma)$ gives $\|w\|_{0,2}\leq C\|u\|_{2,2}$ to complete the proof.
\end{proof}
\begin{lemma}\label{lem:SurfaceLagrangeInfSup}
	There is $\alpha>0$ such that for any $(u,w)\in X_0 \times Y $ it holds that
	\[
		\alpha(\|u\|_{1,q} + \|w\|_{1,p} ) \leq \sup_{(\eta,\xi) \in X_0 \times Y} \frac{c(u,\eta) + b(\eta,w) + b(u,\xi) - m(w,\xi)}{\|\eta\|_{1,q} + \|\xi\|_{1,p} }.
	\]
\end{lemma}
\begin{proof}
It is sufficient \cite{ErnGue13} to show for any $(\alpha,\beta,\tilde{Z},\tilde{f},\tilde{g}) \in \R\times\R\times\R^N\times W^{1,q}(\Gamma)^* \times W^{1,p}(\Gamma)^*$, the following system has a unique solution
\begin{equation}\label{eq:LemmaInfSupDummy}\begin{split}
	c(u,\eta) + b(\eta,w) + (\lambda,T\eta)_{\R^N} &= \langle \tilde{f}, \eta \rangle \quad \forall \eta \in X,
	\\
	b(u,\xi) - m(w,\xi) &= \langle \tilde{g}, \xi \rangle \,\quad \forall \xi \in Y,
	\\
	Tu &= \tilde{Z}, \\
	\int_\Gamma u &= \alpha, \\
	\int_\Gamma w &= \beta. \end{split}
\end{equation}
We are able to find unique $u_1 \in H^2(\Gamma)$ with $Tu_1 = \tilde{Z}$, $\int_\Gamma u =\alpha$ and $a(u_1,v) = \langle \tilde{f},v\rangle$ for all $v \in H^2(\Gamma)$ with $\int_\Gamma v = 0$ and $Tv =0$.
In particular, as discussed in Subsection \ref{subsec:NearSphericalBiomembrane},
\[
	a(u_1,v) + (\lambda_1,Tv)_{\R^n} +\bar{p}_1 \int_\Gamma v = \langle \tilde{f},v\rangle ~~ \forall v \in H^2(\Gamma).
\]
From this formulation, it then follows that $-\Delta_\Gamma u_1 \in W^{1,p}(\Gamma)$ as in Proposition \ref{prop:regularityOfSol}.
Defining $w_1 := -\Delta_\Gamma u_1 + u_1 + \frac{\beta - \alpha}{|\Gamma|} \in W^{1,p}(\Gamma)$, we see that we have a solution to the problem
\begin{equation}\label{eq:LemmaInfSupDummy2}\begin{split}
	c(u_1,\eta) + b(\eta,w_1) + (\lambda_1,T\eta)_{\R^N} &= \langle \tilde{f}, \eta \rangle \quad \forall \eta \in X,
	\\
	b(u_1,\xi) - m(w_1,\xi) &= 0 \qquad\quad \forall \xi \in Y,
	\\
	Tu_1 &= \tilde{Z}, \\
	\int_\Gamma u_1 &= \alpha,\\
	\int_\Gamma w_1 &= \beta.
	\end{split}
\end{equation}
We still require an inhomogeneity for the second equation.
For $G\colon (W^{1,p}(\Gamma))^*\to W^{1,q}(\Gamma)$ such that $b(G\tilde{g},\xi) = \langle \tilde{g}, \xi\rangle~~\forall \xi \in W^{1,p}(\Gamma)$, we may find $(\tilde{u},w,\lambda)$ such that
\begin{align*}
c(\tilde{u},\eta) + b(\eta,w) + (T\eta, \lambda) &= \langle \tilde{f},\eta \rangle - c(G\tilde{g}, \eta)\quad \forall \eta \in X,
\\
b(\tilde{u},\xi) - m(w,\xi) &= 0 \quad \forall \xi \in Y,
\\
T\tilde{u} &= \tilde{Z} - TG\tilde{g},
\\
\int_\Gamma \tilde{u} &= \alpha - \int_\Gamma G\tilde{g},
\\
\int_\Gamma w &= \beta,
\end{align*}
by defining $u := \tilde{u} + G\tilde{g}$ we have $(u,w,\lambda)$ uniquely satisfies \eqref{eq:LemmaInfSupDummy},
which completes the result.
\end{proof}

\subsubsection{Well posedness}

We are now able to prove well-posedness of the following problems which represent  a generalisation of Problems \ref{ProblemLagrangeApplication} and \ref{ProblemPenaltyApplication} by the inclusion of $g$ in the right hand side of the second equation.

\begin{problem}\label{prob:LagrangeSurface}
	Given $f \in (W^{1,q}(\Gamma))^*$ and $g \in (W^{1,p}(\Gamma))^*$, find $(u,w,\lambda)\in X \times Y \times \R^N$ such that
	\begin{align*}
		c(u,\eta) + b(\eta,w) + (T\eta,\lambda )_{\R^N} &= \langle f, \eta \rangle \quad &\forall \eta \in X,
		\\
		b(u,\xi) - m(w,\xi) &= \langle g, \xi\rangle \quad &\forall \xi \in Y,
		\\
		Tu &= Z.
	\end{align*}
\end{problem}
\begin{problem}\label{Prob:SurfacePenalty}
	Given $f\in (W^{1,q}(\Gamma))^*$, $g\in (W^{1,p}(\Gamma))^*$ and $\epsilon>0$, find $(u^\epsilon,w^\epsilon) \in X\times Y$ such that
	\begin{align*}
		c(u^\epsilon,\eta) + b(\eta,w^\epsilon) + \frac{1}{\epsilon}(Tu^\epsilon,T\eta)_{\R^N} &= \langle f, \eta \rangle + \frac{1}{\epsilon}(Z,T\eta)_{\R^N}\quad &\forall \eta \in X,
		\\
		b(u^\epsilon,\xi) - m(w^\epsilon,\xi) &= \langle g, \xi\rangle \quad &\forall \xi \in Y.
	\end{align*}
\end{problem}
\begin{remark}
If one were to suppose that Problems \ref{ProblemLagrangeApplication} and \ref{ProblemPenaltyApplication} have right hand side forcing $F \in \left(H^2(\Gamma) \right)^*$, where $F = f - \Delta_\Gamma g + g$.
Then, formally, the splitting $w:= -\Delta_\Gamma u + u - g$ gives rise to Problems \ref{prob:LagrangeSurface} and \ref{Prob:SurfacePenalty}.
This may be interpreted as a decomposition of the data $F$ into "smooth" and "singular" components.
\end{remark}

\begin{theorem}
	There is a unique solution to Problem \ref{Prob:SurfacePenalty}.
\end{theorem}
\begin{proof}
	The proof is an application of Theorem \ref{Thm:PenaltyProblem} as we have shown Assumptions \ref{ass:AbstractCoercivity}, \ref{InfSupSOSCoercivity} and \ref{ass:discreteInfSupB} in Propositions \ref{prop:SurfaceCoercivity} and \ref{lem:surfaceMeanValueWork}.
\end{proof}

We now wish to show the appropriate assumptions for the well-posedness of the following problem.


\begin{theorem}
	There is unique solution to Problem \ref{prob:LagrangeSurface}.
\end{theorem}
\begin{proof}
	This follows directly from Lemma \ref{lem:SurfaceLagrangeInfSup} as this proves that Assumption \ref{kernelInfSup} holds true.
\end{proof}

From Propositions \ref{Prop:FirstAbstractEstimate} and \ref{prop:increasedEpsilonConvergence}, we have the following result.
\begin{corollary}\label{cor:ApplicationEpsilonConvergence}
Let $(u,w,\lambda)$ solve Problem \ref{prob:LagrangeSurface} and $(u^\epsilon,w^\epsilon)$ solve Problem \ref{Prob:SurfacePenalty}.
Then there is $C>0$ such that
\[
	\|u-u^\epsilon\|_X + \|w-w^\epsilon\|_L+ \sqrt{\epsilon}\|w-w^\epsilon\|_Y + \sqrt{\epsilon}\|{\epsilon}^{-1} (Tu^\epsilon - Z)-\lambda\|_{\R^N} \leq C\epsilon\|\lambda\|_{\R^N}.
\]
\end{corollary}

\subsection{A near flat biomembrane}\label{subsec:NearFlatBiomembrane}
We give a flavour of how this same theory may be applied to the case of the Monge-Gauge.
The Monge-Gauge is studied in \cite{EllGraHob16} and it is noted that it is a geometric linearisation of the Canham-Helfrich energy or indeed, formally, the limit as $R \to \infty$ in the $a$ given in Definition \ref{def:4OrderSurface}.
Let $\Omega$ be a smooth bounded domain in $\R^2$, and fix $\kappa >0$ and $\sigma \geq 0$.
Fix $N \in \mathbb{N}$ and distinct $X_i \in \Omega$ for $i=1,...,N$ so that $S = \R^N$ and $T$ is the evaluation map at these $N$ points.

For this flat problem, we consider the Monge-Gauge energy \cite{EllGraHob16}.
The numerical analysis for this has been considered in \cite{GraKie18} for finite size particles with constraints on closed curves using a penalty method.
The authors make use of higher order $H^2$ conforming finite elements so do not need to split the equation.
\subsubsection{Fourth order formulation}
\begin{definition}
	Define $a\colon H^2(\Omega)\times H^2(\Omega) \to \R$ by
	\[
		a(u,v) := \int_\Omega \kappa \Delta u \Delta v + \sigma \nabla u \cdot \nabla v,
	\]
	$T\colon C(\Omega) \to \R^N$ by
	\[
		Tu = (u(X_i))_{i=1}^N,
	\]
	and for any $\epsilon >0$,  $a_\epsilon\colon H^2(\Omega)\times H^2(\Omega) \to \R$ by
	\[
		a_\epsilon(u,v) := a(u,v) + \frac{1}{\epsilon} (Tu , Tv)_{\R^N}. 
	\]
\end{definition}
It may be seen \cite{EllGraHob16} that $a$ is coercive over $V:=H^2(\Omega) \cap H^1_0(\Omega)$, which corresponds to so called Navier boundary conditions, which we consider here.
\begin{problem}\label{problem:flat4Lagrange}
	Given $f \in (H^2(\Omega))^*$, find $u \in V$ minimising $\frac{1}{2}a(u,u) - \langle f, u \rangle$ subject to $u(X_i) = Z_i$ for $i=1,...,N$.
	This has variational formulation to find $u \in V$ such that $u(X_i) = Z_i$ for $i=1,...,N$ and 
	\[
		a(u,\eta) = \langle f, \eta\rangle ~~\forall \eta \in V : T\eta =0.
	\]
\end{problem}
\begin{problem}\label{problem:flat4Penalty}
Given $f \in (H^2(\Omega))^*$, find $u^\epsilon \in V$ minimising $\frac{1}{2}a(u^\epsilon,u^\epsilon) + \frac{1}{2\epsilon} | Tu^\epsilon - Z|^2 - \langle f, u^\epsilon\rangle$.
	This has variational formulation to find $u^\epsilon \in V$ such that
	\[
		a_\epsilon(u^\epsilon,\eta) = \frac{1}{\epsilon}(Z,T\eta)_{\R^N} + \langle f,\eta\rangle ~~ \forall \eta \in V.
	\]
\end{problem}
\begin{theorem}
	There are unique solutions to both Problems \ref{problem:flat4Lagrange} and \ref{problem:flat4Penalty}.
\end{theorem}
\begin{proof}
This is shown in \cite{EllGraHob16} by making use of the Lax-Milgram theorem with the coercivity of $a$ over $V$.
\end{proof}
For $f=0$, these are the membrane problem studied in \cite{EllGraHob16,GraKie18}.
In very much the same way as the preceding subsection, one may see that the point constraint problem can be written as the following PDE in distribution
\begin{align*}
	\kappa \Delta^2 u - \sigma \Delta u + \sum_{i=1}^N \lambda_i \delta_{X_i} &= f ~ \mbox{in}~ \Omega,
	\\
	u(X_i) &= Z_i ~ \mbox{for} ~ i=1,...,N,
	\\
	u|_{\partial \Omega} = \Delta u |_{\partial \Omega} &= 0.
\end{align*}
With the penalty problem having the distributional PDE, 
\begin{align*}
	\kappa \Delta^2 u^\epsilon - \sigma \Delta u^\epsilon + \frac{1}{\epsilon} \sum_{i=1}^N u^\epsilon\delta_{X_i} &= f + \frac{1}{\epsilon} \sum_{i=1}^N Z_i \delta_{X_i}~ \mbox{in} ~ \Omega
	\\
	u|_{\partial\Omega} = \Delta u |_{\partial\Omega} &= 0.
\end{align*}
\subsubsection{Second order splitting applied to this fourth order problem}
\begin{definition}\label{def:SOSFlat}
	Let $\infty > q > 2 > p > 1$, then we define $X = W^{1,q}_0(\Omega)$, $Y=W^{1,p}_0(\Omega)$, $L = L^2(\Omega)$ and $ S= \R^N$, with the operators
	\begin{align*}
		c\colon& X \times X \to \R,
		\\
		b\colon& X \times Y \to \R,
		\\
		m \colon& L \times L \to \R,
	\end{align*}
	given by
	\begin{align*}
		c(u,\eta) =& \int_\Omega \left( \frac{\sigma}{\kappa} -2\right)\nabla u \cdot \nabla \eta - u \eta,
		\\
		b(\eta,\xi) =& \int_\Omega \nabla \eta \cdot \nabla \xi + \eta \xi,
		\\
		m(\eta,\xi) =& \int_\Omega \eta \xi.
	\end{align*}
\end{definition}
This definition allows us to pose the problems for this flat case. Note the generalisation of Problems \ref{problem:flat4Lagrange} and \ref{problem:flat4Penalty} by the inclusion of $g$ in the right hand side of the second equations.
\begin{problem}\label{prob:LagrangeFlat}
	Given $f\in (W^{1,q}(\Omega))^*$, $g\in (W^{1,p}(\Omega))^*$, find $(u,w,\lambda)\in X \times Y \times \R^N$ such that
	\begin{align*}
		c(u,\eta) + b(\eta,w) + (T\eta,\lambda )_{\R^N} &= \langle f, \eta \rangle \quad &\forall \eta \in X,
		\\
		b(u,\xi) - m(w,\xi) &= \langle g, \xi\rangle \quad &\forall \xi \in Y,
		\\
		Tu &= Z.
	\end{align*}
\end{problem}
\begin{problem}\label{Prob:PenaltyFlat}
	Given $f\in (W^{1,q}(\Omega))^*$, $g\in (W^{1,p}(\Omega))^*$, find $(u^\epsilon,w^\epsilon) \in X\times Y$ such that
	\begin{align*}
		c(u^\epsilon,\eta) + b(\eta,w^\epsilon) + \frac{1}{\epsilon}(Tu^\epsilon,T\eta)_{\R^N} &= \langle f, \eta \rangle + \frac{1}{\epsilon}(Z,T\eta)_{\R^N}\quad &\forall \eta \in X,
		\\
		b(u^\epsilon,\xi) - m(w^\epsilon,\xi) &= \langle g, \xi\rangle \quad &\forall \xi \in Y.
	\end{align*}
\end{problem}
Checking the required assumptions, Asssumptions \ref{ass:AbstractCoercivity}, \ref{InfSupSOSCoercivity}, \ref{ass:discreteInfSupB} and \ref{kernelInfSup} hold almost identically as in Subsection \ref{subsec:NearSphericalBiomembrane} and gives the following theorem and corollary.
\begin{theorem}
	There are unique solutions to Problems \ref{Prob:PenaltyFlat} and \ref{prob:LagrangeFlat}.
\end{theorem}
\begin{corollary}\label{cor:ApplicationEpsilonConvergenceFlat}
Let $(u,w,\lambda)$ solve Problem \ref{prob:LagrangeFlat} and $(u^\epsilon,w^\epsilon)$ solve Problem \ref{Prob:PenaltyFlat}.
Then there is $C>0$ such that
\[
	\|u-u^\epsilon\|_X + \|w-w^\epsilon\|_L+ \sqrt{\epsilon}\|w-w^\epsilon\|_Y + \sqrt{\epsilon}\|{\epsilon}^{-1} (Tu^\epsilon - Z)-\lambda\|_{\R^N} \leq C\|\lambda\|_{\R^N}\epsilon.
\]
\end{corollary}

\section{Finite element approximation of the membrane problems with point constraints}

\subsection{A near spherical biomembrane}\label{subsec:NumAnaSurface}
\begin{definition}\label{def:applicationApproxBil}
Define the following bilinear forms on the discrete function space
\begin{align*}
	c_h(u_h,\eta_h) =& \int_{\Gamma_h} \left(\frac{\sigma}{\kappa}-2-\frac{2}{R^2}\right)\nabla_{\Gamma_h}u_h \cdot \nabla_{\Gamma_h} \eta_h - \left(1+\frac{2\sigma}{\kappa R^2}\right) u_h \eta_h,
	\\
	b_h(\eta_h,\xi_h) =& \int_{\Gamma_h} \nabla_{\Gamma_h}\eta_h\cdot \nabla_{\Gamma_h}\xi +\eta_h\xi_h,
	\\
	m_h(\eta_h,\xi_h) =& \int_{\Gamma_h} \eta_h\xi_h,
\end{align*}
with $T_h \eta_h := T \eta_h^l$.
We take $f_h,g_h \in (\mathcal{S}_h^l)^*$ to satisfy $\langle f_h,\eta_h \rangle = \langle f, \eta_h^l\rangle$, $\langle g_h,\xi_h \rangle = \langle g, \xi_h^l\rangle$.
\end{definition}
We now verify Assumption \ref{ass:DiscreteLagrange}.
\begin{proposition}\label{prop:surfaceSmallInf-Sup}
	There is $C>0$ such that
	\[
		C(\|w_h^l\|_{1,p} + \|\lambda_h\|_{\R^N})
		\leq
		\sup_{(\eta_h^l,\xi_h^l) \in (\mathcal{S}_h^l\cap X)\times (\mathcal{S}_h^l \cap Y) )} \frac{b(\eta_h^l,w_h^l) + (\lambda_h,T\eta_h^l) + m(w_h^l,\xi_h^l))}{\|\eta_h^l\|_{1,q} + \|\xi_h^l\|_{1,p}}
	\]
\end{proposition}
\begin{proof}
Well-posedness of Problem \ref{prob:LagrangeSurface} gives that for any $(u,w,\lambda) \in X \times Y \times \R^N$, there is $C>0$ such that
\[
C(\|u\|_{1,q} + \|w\|_{1,q} + \|\lambda\|_{\R^N} ) \leq \sup_{\substack{(\eta,\xi,\chi)\\ \in X \times Y \times \R^N}} \frac{c(u,\eta) + b(\eta,w) + (T\eta,\lambda) + b(u,\xi) - m(w,\xi) + (Tu,\chi)_{\R^N} }{\|\eta\|_{1,q} + \|\xi\|_{1,p} + \|\chi\|_{\R^N} }.
\]
By applying this with $(0,y_h^l,\chi_h)\in X\times Y\cap \mathcal{S}_h^l \times \R^N$,
\begin{align*}
	\|y_h^l\|_{1,p} + \|\chi_h\|_{\R^N}
	\leq&
	C\sup_{(\eta,\xi) \in X\times Y} \frac{b(\eta,y_h^l) + (T\eta,\chi_h) + m(y_h^l,\xi)}{\|\eta\|_{1,q} + \|\xi\|_{1,p}}
	\\
	=&
	C\sup_{(\eta,\xi) \in X\times Y} \frac{b(\Pi_h\eta,y_h^l) + (T\Pi \eta,\chi_h) + (T(\eta-\Pi_h\eta),\chi_h) + m(y_h^l,P_h\xi)}{\|\eta\|_{1,q} + \|\xi\|_{1,p}}
	\\
	\leq&
	C\sup_{(\eta,\xi) \in X\times Y} \frac{b(\Pi_h\eta,y_h^l) + (T\Pi_h\eta,\chi_h) + m(y_h^l,P_h\xi)}{\|\Pi_h\eta\|_{1,q} + \|P_h \xi\|_{1,p}} + C' h^{1-2/q}|\log(h)| \|\chi_h\|_{\R^N}.
\end{align*}
Where $P_h$ is the $L^2(\Gamma)$ projection, the $\log$ term appears from $\|\eta-\Pi_h\eta\|_{0,\infty} \leq C|\log(h)| \|\eta-I_h^l\eta\|_{0,\infty}$ \cite{Sch80}, the $h^{1-2/q}$ follows from interpolation inequalities and $P_h$ is a bounded operator from $W^{1,p}(\Gamma)$ to itself \cite{BraPasSte02}.
Thus for sufficiently small $h$, this completes the proof.
\end{proof}

\begin{problem}\label{prob:SurfaceDiscreteLagrange}
	Find $(u_h,w_h,\lambda_h) \in \mathcal{S}_h\times\mathcal{S}_h\times\R^N$ such that $\int_{\Gamma_h} u_h = 0$ and
	\begin{align*}
		c_h(u_h,\eta_h) + b_h(\eta_h,w_h) + (T_h \eta_h, \lambda_h)_{\R^N} &= \langle f_h, \eta_h\rangle	\quad \forall \eta_h \in \mathcal{S}_h : \int_{\Gamma_h} \eta_h = 0,
		\\
		b_h(u_h,\xi_h) - m_h(w_h,\eta_h) &= \langle g_h,\xi_h\rangle	\quad \forall \xi_h \in \mathcal{S}_h,
		\\
		T_h u_h &= Z.
	\end{align*}
\end{problem}
\begin{theorem}\label{thm:SurfaceDiscreteLagrange}
There is unique solution to Problem \ref{prob:SurfaceDiscreteLagrange}.
Moreover, for $g \in L^2(\Gamma)$ it holds that
\[
	\|u-u_h^l\|_{1,2} + \|w-w_h^l\|_{0,2} \leq Ch^{2/q} (\|f\|_{-1,p} + \|g\|_{0,2} + \|Z\|_{\R^N}).
\]
\end{theorem}
One might hope that estimate follows from Theorem \ref{ThmDiscreteLagrangeProblem}, Corollary \ref{Corollary:DiscreteLagrangeEstimates} and Proposition \ref{prop:DiscreteLagrangeEstimates}.
However, it is possible to see that due to our choice of $T$ the maximum regularity one might expect is $\bar{X}=H^2(\Gamma),$ $\hat{X}=W^{3,p}(\Gamma)$, $\bar{Y}=\hat{Y}=W^{1,p}(\Gamma)$, and $\bar{S}=\hat{S}=S$, which would give $\alpha=\beta=0$ in the context of Proposition \ref{prop:DiscreteLagrangeEstimates}.
As such we require a different method, the idea is to, in the proof of Proposition \ref{prop:DiscreteLagrangeEstimates}, pick $\xi_h$ to be $\Pi_h w$ which gives that the term which would depend on $\|\phi-\xi_h^l\|_Y$ vanishes.
We also address the fact that the typical lift map from the discrete surface to the continuous surface will not, in general, preserve the integral of functions.
\begin{proof}
	The existence follows from Theorem \ref{ThmDiscreteLagrangeProblem}.
	For the estimate, consider $(\psi,\phi,\chi)\in X\times W^{1,p}(\Gamma) \times \R^N$ such that
	\begin{align*}
		c(\eta,\psi) + b(\eta,\phi) + (T\eta,\chi) &= ( u-u_h^l,\eta)_{H^1(\Gamma)} \quad \forall \eta \in X,
		\\
		b(\psi,\xi) - m(\phi,\xi) &= ( w-w_h^l,\xi )_{L^2(\Gamma)} \quad \forall \xi \in W^{1,p}(\Gamma),
		\\
		T\psi &= \lambda - \lambda_h.
	\end{align*}
	This has unique solution with $\|\psi\|_{2,2} + \|\phi\|_{1,p} + \|\chi\|_{\R^N} \leq C (\|u-u_h^l\|_{1,2} + \|w-w_h^l\|_{0,2} + \|\lambda-\lambda_h\|_{\R^N})$.
	Testing this system with $(u-u_h^l +[u_h^l],w-w_h^l,\lambda-\lambda_h)$, where $[v]:=\frac{1}{|\Gamma|}\int_\Gamma v$ is the average value of $v$, gives
	\begin{align*}
		\|u-u_h^l\|_{1,2}^2 + ( u-u_h^l,[u_h^l])_{H^1(\Gamma)} +& \|w-w_h^l\|_{0,2}^2 + \|\lambda-\lambda_h\|_{\R^N}^2
		\\=&
		c(u-u_h^l+[u_h^l], \psi) + b(\psi, w-w_h^l) + (T\psi,\lambda-\lambda_h)_{\R^N}
		\\&+
		b(u-u_h^l+[u_h^l],\phi) - m(w-w_h^l,\phi)
		+
		(T(u-u_h^l),\chi)_{\R^N}.
	\end{align*}
	 The final term here is $0$ when $h$ is sufficiently small that $T_h u_h = Tu_h^l = Tu =Z$, and it holds that $|[u_h^l]| \leq Ch^2\|u_h^l\|_{0,2}$.
	We consider
	\begin{align*}
		b(u-u_h^l,\phi) - m(w-w_h^l,\phi)
		=&
		\langle g, \phi-\Pi_h \phi\rangle + \langle g,\Pi_h \phi\rangle - \langle g_h, \Pi_h \phi ^{-l} \rangle
		\\
		&+
		m(w_h^l,\phi-\Pi_h\phi) + b_h(u_h,\Pi_h \phi^{-l}) - b(u_h^l,\Pi_h \phi)
		\\
		&+m(w_h^l,\Pi_h \phi) - m_h(w_h,\Pi_h \phi^{-l})
		\\
		\leq&
		C \big[ \left(\|g\|_{0,2}+\|w_h^l\|_{0,2}\right)\|\phi-\Pi_h\phi\|_{0,2}
		\\
		&+ h^2 \left(\|g\|_{0,2} + \|u_h^l\|_{1,q}+ \|w_h^l\|_{0,2}\right)\|\phi\|_{1,p}\big]
		\\
		\leq&
		C\left[ h^{2/q} (\|g\|_{0,2}+ \|w_h^l\|_{0,2} ) + h^2 (\|g\|_{0,2} + \|u_h^l\|_{1,q} + \|w_h^l\|_{0,2}) \right] \|\phi\|_{1,p}.
	\end{align*}
	For the remaining terms,
	\begin{align*}
		c(u-u_h^l,\psi) &+ b(\psi,w-w_h^l) + (T\psi,\lambda-\lambda_h)_{\R^N}
		\\
		=&
		\langle f,\psi\rangle - c(u_h^l,\psi-I_h^l \psi) - b(\psi-I_h^l\psi,w_h^l) - (T(\psi-I_h^l \psi),\lambda_h)_{\R^N}
		\\
		&-c(u_h^l,I_h^l\psi) - b(I_h^l\psi,w_h^l) - (TI_h^l\psi,\lambda_h)_{\R^N} -\langle f_h,I_h\psi-[I_h\psi]\rangle
		\\
		&+ c_h(u_h,I_h\psi-[I_h\psi]) + b_h(I_h\psi-[I_h\psi],w_h) + (T_h (I_h\psi-[I_h\psi]),\lambda_h)_{\R^N}
		\\
		\leq&
		C\big[ (\|f\|_{-1,p}+ \|u_h^l\|_{1,q}+ \|w_h^l\|_{1,p}+\|\lambda_h\|_{\R^N}) \|\psi-I_h^l\psi\|_{1,q}
		\\
		&+ h^2\|I_h^l\psi\|_{1,q}(\|u_h^l\|_{1,q} +\|w_h^l\|_{1,p} + \|\lambda_h\|_{\R^N} + \|f\|_{-1,p})
		\\
		&+ [I_h\psi](\|u_h^l\|_{1,q} +\|w_h^l\|_{1,p} + \|\lambda_h\|_{\R^N} + \|f\|_{-1,p})\big]\\
		\leq&
		C\big[ h^{2/q}(\|f\|_{-1,p}+ \|u_h^l\|_{1,q}+ \|w_h^l\|_{1,p}+\|\lambda_h\|_{\R^N}) \|\psi\|_{2,2}
		\\
		&+ h^2\|I_h^l\psi\|_{1,q}(\|u_h^l\|_{1,q} +\|w_h^l\|_{1,p} + \|\lambda_h\|_{\R^N} + \|f\|_{-1,p})
		\\
		&+ [I_h\psi](\|u_h^l\|_{1,q} +\|w_h^l\|_{1,p} + \|\lambda_h\|_{\R^N} + \|f\|_{-1,p})\big],
	\end{align*}
	where we may see that $[I_h\psi] \leq C h^2\|\psi\|_{2,2}$, and we have used
	\[
		\langle f,I_h^l\psi\rangle - \langle f_h,I_h\psi\rangle\leq Ch^2 \|f\|_{-1,p}\|I_h^l \psi\|_{1,q},
		\]
	which follows from the estimates on $m$ and $m_h$ and using density.
	Hence we have, after using Young's inequality on the additional left hand side term,
	\begin{align*}
		\|u-u_h^l\|_{1,2}^2 +& \|w-w_h^l\|_{0,2}^2 + \|\lambda-\lambda_h\|_{\R^N}^2
		\\
		\leq&
		Ch^{2/q} (\|g\|_{0,2}+ \|f\|_{-1,p} + \|u_h^l\|_{1,q} + \|w_h^l\|_{1,p} +\|\lambda_h\|_{\R^N})(\|\psi\|_{2,2} + \|\phi\|_{1,p}),
	\end{align*}
	which after using $\|u_h^l\|_{1,q} + \|w_h^l\|_{1,p} +\|\lambda_h\|_{\R^N} \leq C (\|g\|_{0,2} + \|f\|_{-1,p} + \|Z\|_{\R^N}),$ and the regularity estimates assumed for $\|\psi\|_{2,2} + \|\phi\|_{1,p}$, gives the result.
\end{proof}

We now wish to improve this estimates to the spaces which are natural to the problem.
\begin{corollary}\label{cor:SurfaceQ}
	Under the assumptions of the above theorem, it holds that 
	\[
		\|u-u_h^l\|_{1,q} \leq C h^{2/q} (\|g\|_{0,2} + \|f\|_{-1,p} + \|Z\|_{\R^N}).
	\]
\end{corollary}
\begin{proof}
	The inf-sup condition in Lemma \ref{lem:surfInfSup} gives,
	\[
		\beta\|u-u_h^l\|_{1,q}
		\leq
		\sup_{\xi \in W^{1,p}(\Gamma)}
		\frac{b(u-u_h^l,\xi)}{\|\xi\|_{1,p} },
	\]
	where we see
	\begin{align*}
		b(u-u_h^l,\xi)
		=& m(w,\xi-\Pi_h\xi) + \langle g,\xi-\Pi_h\xi\rangle + m(w-w_h^l,\Pi_h\xi)
		\\
		&+
		m(w_h^l,\Pi_h\xi) - b(u_h^l,\Pi_h\xi) + \langle g,\Pi_h \xi \rangle
		\\
		&- m_h(w_h,(\Pi_h\xi)^{-l} ) + b_h(u_h,(\Pi_h\xi)^{-l}) - \langle g_h,(\Pi_h\xi)^{-l}\rangle
		\\
		\leq&
		Ch^{2/q}(\|w\|_{0,2} + \|g\|_{0,2})\|\xi\|_{1,p} + C\|w-w_h^l\|_{0,2} \|\xi\|_{1,p}.
	\end{align*}
	The estimate shown for $\|w-w_h^l\|_{0,2}$ in Theorem \ref{thm:SurfaceDiscreteLagrange} completes the result.
\end{proof}
\begin{corollary}\label{cor:SurfaceP}
	In addition to the assumptions of the above, assume $f \in C(\Gamma)^*$, it then holds,
	\[
		\|w-w_h^l\|_{1,p} \leq C h^{\min(2/p-1,2/q)}|\log(h)| (\|f\|_{C(\Gamma)^*} + \|g\|_{0,2} + \|Z\|_{\R^N}).
	\]
\end{corollary}
\begin{proof}
	The inf-sup condition from Proposition \ref{lem:surfaceMeanValueWork} gives,
	\[
		\bar{\gamma}\|w-w_h^l - [w] + [w_h^l]\|_{1,p} \leq\sup_{\eta\in X}\frac{b(\eta,w-w_h^l -[w]+[w_h^l])}{\|\eta\|_{1,q}} = \sup_{\eta\in X}\frac{b(\eta,w-w_h^l)}{\|\eta\|_{1,q}}.
	\]
	Thus we calculate
	\begin{align*}
		b(\eta,w-w_h^l)
		=&
		\langle f, \eta- \Pi_h\eta\rangle - c(u,\eta-\Pi_h\eta) - (\lambda,T(\eta-\Pi_h \eta) )_{\R^N}
		\\
		&+ \langle f,\Pi_h\eta\rangle + c(u_h^l-u,\Pi_h\eta) +(\lambda_h-\lambda,T\Pi_h\eta)_{\R^N} - (\lambda_h,T\Pi_h\eta)_{\R^N}
		\\
		&- c(u_h^l,\Pi_h\eta)  - b(\Pi_h\eta,w_h^l)
		\\
		&+
		c_h(u_h,\Pi_h^{-l}\eta) + b_h(\Pi_h^{-l}\eta,w_h) + (\lambda_h,T_h\Pi_h^{-l}\eta)_{\R^N} - \langle f_h,\Pi_h^{-l}\eta\rangle.
	\end{align*}
	For our particular choice of $c$, it holds that 
	$c(u,\eta) \leq C \|u\|_{2,2}\|\eta\|_{0,2}$.
	For $ f \in C(\Gamma)^*$ it holds $\langle f,\eta \rangle \leq \|f\|_{C(\Gamma)^*} \|\eta\|_{0,\infty}$.
	Furthermore, we have that $T_h\Pi_h^{-l} \eta = T \Pi_h \eta$ by definition.
	From \cite{Sch80} $\|\eta-\Pi_h\eta\|_{0,\infty} \leq C |\log(h)|\|\eta-I_h^l\eta\|_{0,\infty}$.
	Together, this gives
	\begin{align*}
		b(\eta,w-w_h^l)
		\leq
		C\|\eta\|_{1,q} &\big[ h\|u\|_{2,2} + |\log(h)|h^{1-2/q} (\|\lambda\|_{\R^N}+ \|f\|_{C(\Gamma)^*})
		\\
		&+ \|u-u_h^l\|_{1,2} + \|\lambda-\lambda_h\|_{\R^N} +h^2(\|f\|_{C(\Gamma)^*} +\|u_h^l\|_{1,q} +\|w_h^l\|_{1,p} ) \big].
	\end{align*}
	Using the estimate in Theorem \ref{thm:SurfaceDiscreteLagrange} for $\|u-u_h^l\|_{1,2} + \|\lambda-\lambda_h\|_{\R^N}$ completes the proof.
\end{proof}
Notice that for $p = 4/3$, this results in almost $h^{\frac{1}{2}}$ convergence.
We now look at the problem with penalty, which follows as in Theorem \ref{thm:UseApproximationProperties}.

\begin{problem}\label{prob:SurfaceDiscretePenalty}
	Find $(u_h^\epsilon,w_h^\epsilon)\in \mathcal{S}_h^l\times\mathcal{S}_h^l$ such that $\int_{\Gamma_h} u_h^\epsilon = 0$ and
	\begin{align*}
		c_h(u_h^\epsilon,\eta_h) + b_h(\eta_h,w_h^\epsilon) + \frac{1}{\epsilon}(T_h u_h^\epsilon,T\eta_h)_{\R^N} &= \langle f_h,\eta_h\rangle + \frac{1}{\epsilon}(Z,T\eta_h)_{\R^N} \quad &\forall \eta_h \in \mathcal{S}_h : \int_{\Gamma_h} \eta_h =0,
		\\
		b_h(u_h^\epsilon,\xi_h) - m(w_h^\epsilon,\xi_h) &= \langle g_h,\xi_h\rangle \quad &\forall \xi_h \in \mathcal{S}_h.
	\end{align*}
\end{problem}

\begin{theorem}\label{thm:surfaceDiscretePenalty}
There is a unique solution to Problem \ref{prob:SurfaceDiscretePenalty}.
Moreover, for $g\in L^2(\Gamma)$, it holds for $f \in W^{1,q}(\Gamma)^*$,
\[
	\|u^\epsilon-(u_h^\epsilon)^l\|_{1,2} +\|w^\epsilon-(w_h^\epsilon)^l\|_{0,2}
	\leq
	C(h^{2/q}+\sqrt{\epsilon})(\|f\|_{-1,p} + \|g\|_{0,2} + \|Z\|_{\R^N})
\]
In particular,
\[
	\|u^\epsilon-(u_h^\epsilon)^l\|_{1,q}
	\leq
	C(h^{2/q}+\sqrt{\epsilon})(\|f\|_{-1,p} + \|g\|_{0,2} + \|Z\|_{\R^N})
\]
and if $f \in C(\Gamma)^*$, then for any $1<p<2<q<\infty$ with $p,q$ conjugate
\[
	\|w^\epsilon-(w_h^\epsilon)^l\|_{1,p}
	\leq
	C (h^{\min(2/p-1,2/q)}|\log(h)|+ \sqrt{\epsilon}) (\|f\|_{C(\Gamma)^*} + \|g\|_{0,2} + \|Z\|_{\R^N}).
\]
\end{theorem}
\begin{proof}
The results follow from the same argument as Theorem \ref{thm:UseApproximationProperties}, where we have the results of Theorem \ref{thm:SurfaceDiscreteLagrange} and Corollaries \ref{cor:SurfaceQ} and \ref{cor:SurfaceP} to give the $h$ estimates and we may see from Corollary \ref{cor:uniformBoundsEpsi} and Lemma \ref{discreteConvergenceInEps} the $\epsilon$ approximation.
\end{proof}

\subsection{A flat biomembrane}\label{subsec:NearFlatMembrane}
We here outline the existence results and estimates as in the preceeding subsection for the case of the nearly flat membrane problem discussed in Subsection \ref{subsec:NearFlatMembrane}.
We make the simplifying assumption that $\Omega$ is convex and polygonal.
\begin{definition}
	Let $\mathcal{T}_h$ be a triangulation of $\Omega$ with $\bar{\Omega} = \bigcup_{K \in \mathcal{T}_h}K$ and  $K^\circ \cap (K')^\circ= \emptyset$ $\forall K,K' \in \mathcal{T}_h$ for $K \neq K'$.
	Define
	\[
		\mathcal{S}_h := \{ \chi \in C(\Omega) \,:\, \chi|_K \in P^1(K) ~\forall K \in \mathcal{T}_h\}.
	\]
	Where $P^1(K)$ is the polynomials of degree 1 or less on $K$.
	The Lagrange basis functions $\phi_i$ of this space are uniquely determined by their values at the so-called Lagrange nodes $q_j$.
	The associated Lagrange interpolation $I_h \colon C(\Omega) \to \mathcal{S}_h$ is given by $I_h f := \sum_{i} f(q_i)\phi_i$.
	We again take the linear functionals as in Definition \ref{def:SOSFlat} and assume that $\mathcal{T}_h$ is a regular triangulation. 
\end{definition}

\begin{problem}\label{problem:discreteFlatLagrange}
Find $(u_h,w_h,\lambda_h) \in \mathcal{S}_h\times \mathcal{S}_h\times \R^N$ such that
\begin{align*}
	c(u_h,\eta_h) + b(\eta_h,w_h) + (T\eta_h , \lambda_h)_{\R^N} =& \langle f_h, \eta_h\rangle~ \forall \eta_h \in \mathcal{S}_h
	\\
	b(u_h,\xi_h) - m(w_h,\xi_h) =& \langle g, \xi_h\rangle ~\forall  \xi_h \in \mathcal{S}_h
	\\
	Tu_h =& Z.
\end{align*}
\end{problem}

\begin{problem}\label{problem:discreteFlatPenalty}
Find $(u_h^\epsilon,w_h^\epsilon) \in \mathcal{S}_h\times \mathcal{S}_h$ such that
\begin{align*}
	c(u^\epsilon_h,\eta_h) + b(\eta_h,w_h^\epsilon) + \frac{1}{\epsilon} (Tu_h^\epsilon,T\eta_h)_{\R^N} =& \langle f ,\eta_h \rangle + \frac{1}{\epsilon} (Z,T\eta_h)_{\R^N} ~~ \forall \eta_h \in \mathcal{S}_h
		\\
	b(u_h,\xi_h) - m(w_h,\xi_h) =& \langle g, \xi_jh\rangle ~\forall  \xi_h \in \mathcal{S}_h
\end{align*}
\end{problem}

\begin{theorem}
	There is a unique solution to Problem \ref{problem:discreteFlatLagrange}.
	Moreover for $g \in L^2(\Omega)$ and any $q>2$, it holds for $f \in (W^{1,q}(\Omega))^*$
	\[
		\|u-u_h\|_{1,q} + \|w-w_h\|_{0,2}
		\leq C(h^{2/q} (\|f\|_{-1,p} + \|g\|_{0,2} + \|Z\|_{\R^N})
	\]
	Furthermore, if $f \in (C(\Omega))^*$, with $\frac{1}{p}+\frac{1}{q}=1$,
	\[
		\|w-w_h\|_{1,p} \leq C (h^{\min(2/p-1,2/q)}|\log(h)|) (\|f\|_{C(\Omega)^*}+ \|g\|_{0,2} + \|Z\|_{\R^N}).
	\]
\end{theorem}

\begin{theorem}
	There is a unique solution to Problem \ref{problem:discreteFlatPenalty}.
	Moreover for $g \in L^2(\Omega)$ and any $q>2$, it holds for $f \in (W^{1,q}(\Omega))^*$
	\[
		\|u^\epsilon-u^\epsilon_h\|_{1,q} + \|w^\epsilon-w^\epsilon_h\|_{0,2}
		\leq C(h^{2/q} +\sqrt{\epsilon}) (\|f\|_{-1,p} + \|g\|_{0,2} + \|Z\|_{\R^N})
	\]
	Furthermore, if $f \in (C(\Omega))^*$, with $\frac{1}{p}+\frac{1}{q}=1$,
	\[
		\|w^\epsilon-w^\epsilon_h\|_{1,p} \leq C (h^{\min(2/p -1,2/q)}|\log(h)| + \sqrt{\epsilon}) (\|f\|_{C(\Omega)^*}+ \|g\|_{0,2} + \|Z\|_{\R^N}).
	\]
\end{theorem}
These results follow from a slight variation of the arguments presented in Subsection \ref{subsec:NumAnaSurface}.


\section{Numerical experiments}\label{Sec:Experiments}
We conclude with some numerical examples.
All of the numerical examples are done for the biomembrane problem as outlined in Section \ref{Sec:SurfaceApplication}.
When we discuss the error at level $h$, we will be referring to the relative error, where we define the relative error between $u$ and $u_h^l$ in norm $\|\cdot\|_W$ to be given by $E_W(h):=\|u-u_h^l\|_W/\|u\|_W$.
The EOC (experimental order of convergence) between levels $h_1$ and $h_2$ is given by $EOC_{W}(h_1,h_2) := \log(E_W(h_1)/E_W(h_2))/\log(h_1/h_2)$.
In the experiments, we will take the EOC to be at the current level and the previous refinement.



All the experiments have been implemented under the Distributed and Unified Numerics Environment (DUNE) \cite{AlkDedKlo16,BlaBurDed16}. 
\subsection{Flat case experiment}
\label{sec:experimentFlat}

The first example is for a flat domain.  Let $\Omega$ be the unit disc in $\R^2$ centred at the origin and 
 $\mathcal P:= \{(0,0),(0.5,0),(-0.5,0),(0,0.5),(0,-0.5)\}$ be 5 distinct points in $\Omega$.
The PDE boundary value problem is
 $$\Delta^2 u = 0~\mbox{in}~ \Omega \setminus \mathcal P$$ such that
  $$~u(X_j) = g(X_j)~ \forall X_j \in \mathcal P, ~u|_{\partial \Omega} = \Delta u|_{\partial\Omega}=0,$$
where $g(x) := 1-|x|^2 + \frac{|x|^2}{2}\log(|x|^2)$. It has the solution $$u(x) = 1-|x|^2 +\frac{|x|^2}{2}\log(|x|^2).$$

This can be viewed as a flat biomembrane problem  with $\kappa = 1$ and $\sigma = 0$.
The coupled second order system  is
\begin{align*}
	2\Delta u - u -\Delta w + w &= 0  ~\mbox{in}~ \Omega\setminus \mathcal P,
	\\
	-\Delta u + u - w &=0  ~\mbox{in}~ \Omega,
	\\
	u|_{\partial \Omega} =  w|_{\partial \Omega} &= 0,
	\\
	u(X_j) &= g(X_j)~ \forall X_j \in \mathcal P.
\end{align*}
As in Subsection \ref{subsec:NearFlatBiomembrane}, we see that for the first equation, this is not posed on the domain $\Omega$, but away from the points being constrained.

The bilinear forms become
\[
	c(u,\eta) = \int_\Omega -2 \nabla u \cdot\nabla \eta - u\eta,~~ b(u,\eta) = \int_\Omega \nabla u \cdot\nabla \eta + u\eta,~~ m(w,\xi) = \int_\Omega w\xi.
\]
Since the problem is posed with homogeneous Navier boundary conditions on the unit disc, we may pose the discrete  
problem on a polygonal domain  $\Omega_h$  which approximates the unit disc from within and extend the finite element spaces to be  $0$ in the 
skin $\Omega\setminus \Omega_h$.
We only calculate the error on the discrete domain, it is clear that the error due to the skin will be sufficiently small that it should not interfere with the calculated interior error. Errors are displayed in the Tables \ref{table:UFlat}, \ref{table:WFlat} and \ref{table:LambdaFlat}. The errors of $\|w-w_h\|_{0,2}$, $\|w-w_h\|_{1,\frac{4}{3}}$ and $\|u-u_h\|_{1,2}$ behave as expected
 from the theory provided in Section \ref{Sec:SurfaceApplication} whereas  the  errors $\|u-u_h\|_{0,2}$ and $\|\lambda-\lambda_h\|_{\R^5}$ converge at a higher rate. 

\begin{table}[h]
\begin{tabular}{|l|l|l|l|l|}
\hline
$h$       & $E_{L^2}$   & $E_{H^1}$  & $EOC_{L^2}$ & $EOC_{H^1}$ \\ \hline
0.420334  & 0.0347383   & 0.132332   & --          & --          \\ \hline
0.221925  & 0.010496    & 0.0724977  & 1.87385     & 0.943152    \\ \hline
0.113732  & 0.00293398  & 0.0377392  & 1.90671     & 0.976601    \\ \hline
0.0575358 & 0.000787479 & 0.0191858  & 1.93016     & 0.992797    \\ \hline
0.0289325 & 0.000206736 & 0.00965453 & 1.94547     & 0.998988    \\ \hline
\end{tabular}
\caption{Errors and experimental orders of convergence for $u-u_h$ in the flat case experiment, Subsection \ref{sec:experimentFlat}.}
\label{table:UFlat}
\end{table}
\begin{table}[h]
\begin{tabular}{|l|l|l|l|l|}
\hline
$h$       & $E_{L^2}$  & $E_{W^{1,\frac{4}{3}}}$ & $EOC_{L^2}$ & $EOC_{W^{1,\frac{4}{3}}}$ \\ \hline
0.420334  & 0.0242845  & 0.435937      & --          & --              \\ \hline
0.221925  & 0.0114147  & 0.316908      & 1.18197     & 0.499267        \\ \hline
0.113732  & 0.0057894  & 0.228215      & 1.01552     & 0.491137        \\ \hline
0.0575358 & 0.00292883 & 0.162944     & 0.99999     & 0.494371        \\ \hline
0.0289325 & 0.00147189 & 0.115812     & 1.0009      & 0.496675        \\ \hline
\end{tabular}
\caption{Errors and experimental orders of convergence for $w-w_h$ in the flat case experiment, Subsection \ref{sec:experimentFlat}.}
\label{table:WFlat}
\end{table}

\begin{table}[h]
\begin{tabular}{|l|l|l|}
\hline
$h$       & $E_{\ell^2}$  & $EOC_{\ell^2}$
\\ \hline
0.420334	&	0.00621158	&	--
\\ \hline
0.221925	&	0.00492358	&	0.363827
\\ \hline
0.113732	&	0.00218382	&	1.2161
\\ \hline
0.0575358	&	0.000763246	&	1.5427
\\ \hline
0.0289325	&	0.000235913	&	1.70795
\\ \hline
\end{tabular}
\caption{Errors and experimental order of convergence for $\lambda-\lambda_h$ in the flat case experiment, Subsection \ref{sec:experimentFlat}.}
\label{table:LambdaFlat}
\end{table}

\subsection{Surface numerical experiment}\label{sec:surfaceNumericalExperiment}
The second numerical example is for the surface of the unit sphere, $\Gamma:= \mathbb{S}(0,1)$. The point constraints are fixed at   the 
six distinct points  $\mathcal P:=\{(\pm 1,0,0),(0,\pm 1,0),(0,0,\pm 1)\}$.
We consider the problem of $\kappa = \sigma = R = 1$ in the forms defined in Definition \ref{def:applicationApproxBil} corresponding  to, is to the PDE boundary value problem, find  $(u,\bar{p})$ such that
\begin{align*}
	\Delta_\Gamma^2 u +\Delta_\Gamma u -2u + \bar{p} &= f - \Delta_\Gamma g + g~\mbox{in}~\Gamma\setminus \mathcal P,
	\\
	u(X_j) &= Z_j ~ \forall X_j\in \mathcal P,
	\\
	\int_\Gamma u &= 0,
\end{align*}
where
\[
	f = 
	9 x_3\log(1-x_3) + 9 x_3 - 2\log(1-x_3) + \frac{1}{2}(5+3\log(4)),~~ ~Z_j=U(X_j) ~~j=1,2,...,6,
\]
\[
	g = -3x_3 \log(1-x_3) -3x_3 - \frac{1}{2}(\log(4)+1)
\]
	and \[
	U = (1-x_3)\log(1-x_3) - \frac{1}{2}(\log(4)-1).
\]
We recall that $\bar{p}$ arises as the Lagrange multiplier associated to the constraint $\int_\Gamma u = 0$, as in Subsection \ref{subsec:NearSphericalBiomembrane}.
The solution to this problem is $u=U$, $\bar{p}=0$. The second order splitting system is taken to be
\begin{align*}
	3\Delta_\Gamma u - 3u -\Delta_\Gamma w + w + \bar{p} &= f ~ \mbox{in}~ \Gamma \setminus \mathcal P,
	\\
	-\Delta_\Gamma u + u - w + \bar{q} &= g ~\mbox{in}~\Gamma,
	\\
	u(X_j) &= Z_j ~ \forall X_j \in \mathcal P,
	\\
	\int_\Gamma u = \int_\Gamma w &= 0,
\end{align*}
where $\bar{q}$ is the Lagrange multiplier due to the constraint on the mean value of $w$.

Thus the forms of Definition \ref{def:applicationBilinearForms} with $\kappa=\sigma=R=1$ are given by
\[
	b(u,\eta)=\int_\Gamma \nabla_\Gamma u \cdot \nabla_\Gamma \eta + u\eta,~
	c(u,\eta) = -3b(u,\eta),~
	m(w,\xi)=\int_\Gamma w\xi.
\]
The well-posedness of the problem follows from Section \ref{Sec:SurfaceApplication} and has solution $u = U$, $w = \log(1-x_3)$.

In these numerical computations, implementation of the point constraints is achieved via penalty with $\epsilon = 10^{-8}$, a value chosen sufficiently small as to play no role in the investigation of the order of convergence  with respect to $h$. The errors are displayed  in Tables \ref{table:USurface} and \ref{table:WSurface}.  They behave similarly to that of   the flat case experiment and are consistent with the theory provided in Section \ref{Sec:SurfaceApplication}.

\begin{table}[h]
\begin{tabular}{|l|l|l|l|l|}
\hline
$h$       & $E_{L^2}$   & $E_{H^1}$  & $EOC_{L^2}$ & $EOC_{H^1}$ \\ \hline
0.311152  & 0.012565    & 0.0841661  & --          & --          \\ \hline
0.156914  & 0.00356525  & 0.042819   & 1.84007     & 0.987187    \\ \hline
0.0786276 & 0.000990194 & 0.0215476  & 1.85403     & 0.993838    \\ \hline
0.0393352 & 0.000276744 & 0.0107968  & 1.84061     & 0.997706    \\ \hline
0.0196703 & 7.88541e-05 & 0.00540193 & 1.81165     & 0.999252    \\ \hline
\end{tabular}
\caption{Errors and experimental orders of convergence for $u-u_h^l$ in the surface numerical experiment, Subsection \ref{sec:surfaceNumericalExperiment}.}
\label{table:USurface}
\end{table}

\begin{table}[h]
\begin{tabular}{|l|l|l|l|l|}
\hline
$h$       & $E_{L^2}$  & $E_{W^{1,\frac{4}{3}}}$ & $EOC_{L^2}$ & $EOC_{W^{1,\frac{4}{3}}}$ \\ \hline
0.311152  & 0.0486308  & 0.236187       & --          & --              \\ \hline
0.156914  & 0.0212111  & 0.165895     & 1.21203     & 0.516039        \\ \hline
0.0786276 & 0.0098867  & 0.118446      & 1.10472     & 0.487569        \\ \hline
0.0393352 & 0.00478169 & 0.0845555     & 1.04879     & 0.486638        \\ \hline
0.0196703 & 0.00235552 & 0.0602071     & 1.02167     & 0.49006        \\ \hline
\end{tabular}
\caption{Errors and experimental orders of convergence for $w-w_h^l$ in the surface numerical experiment, Subsection \ref{sec:surfaceNumericalExperiment}.}
\label{table:WSurface}
\end{table}

\subsection{Penalty experiment} \label{sec:PenaltyExperiment}
 We now fix $h$ to be sufficiently small that it should contribute little error and  take a sequence of $\epsilon$ which will tend to $0$.
For simplicity, we consider the same experiment as in Subsection \ref{sec:surfaceNumericalExperiment}.
Where previously the quantities $E$ and $EOC$ have been functions of $h$, they will now be functions of $\epsilon$ in the expected way.
The grid is fixed to be the smallest grid used in the previous experiment with $h= 0.0196703$.
In Tables \ref{table:UPenalty}, \ref{table:WPenalty} and \ref{table:LambdaPenalty} we see that the errors are consistent with the results of Corollary \ref{cor:ApplicationEpsilonConvergence},  Theorem \ref{thm:SurfaceDiscreteLagrange} and Theorem \ref{thm:surfaceDiscretePenalty}. 

\begin{table}[h]
\begin{tabular}{|l|l|l|l|l|}
\hline
$\epsilon$ & $E_{L^2}$ & $E_{H^1}$ & $EOC_{L^2}$ & $EOC_{H^1}$ \\ \hline
0.2        & 0.178146  & 0.173119  & --          & --          \\ \hline
0.1        & 0.09091   & 0.0910002 & 0.970551    & 0.970551    \\ \hline
0.05       & 0.0462214 & 0.0477307 & 0.975878    & 0.930951    \\ \hline
0.025      & 0.0233842 & 0.025107  & 0.983028    & 0.926831    \\ \hline
0.0125     & 0.0117769 & 0.013647  & 0.989572    & 0.879506    \\ \hline
\end{tabular}
\caption{Errors and experimental orders of convergence for $u-(u^\epsilon_h)^l$ in the numerical experiment, Subsection \ref{sec:PenaltyExperiment}.}
\label{table:UPenalty}
\end{table}

\begin{table}[h]
\begin{tabular}{|l|l|l|l|l|}
\hline
$\epsilon$ & $E_{L^2}$ & $E_{W^{1,\frac{4}{3}}}$ & $EOC_{L^2}$ & $EOC_{W^{1,\frac{4}{3}}}$ \\ \hline
0.2        & 0.337878  & 0.441999	& --          & --              \\ \hline
0.1        & 0.191338  & 0.29658	& 0.820381    & 0.575622        \\ \hline
0.05       & 0.107691  & 0.19941	& 0.829231    & 0.572684        \\ \hline
0.025      & 0.0592661 & 0.134285	& 0.861611    & 0.570441        \\ \hline
0.0125     & 0.0316392 & 0.0952299	& 0.905497    & 0.495808        \\ \hline
\end{tabular}
\caption{Errors and experimental orders of convergence for $w-(w^\epsilon_h)^l$ in the numerical experiment, Subsection \ref{sec:PenaltyExperiment}.}
\label{table:WPenalty}
\end{table}

\begin{table}[h]
\begin{tabular}{|l|l|l|}
\hline
$\epsilon$	& $E_{\ell^2}$  & $EOC_{\ell^2}$
\\ \hline
0.2			&	0.568092		&	--
\\ \hline
0.1			&	0.421624		&	0.430169
\\ \hline
0.05		&	0.287481		&	0.552491
\\ \hline
0.025		&	0.177819		&	0.693052
\\ \hline
0.0125		&	0.101217		&	0.81296
\\ \hline
\end{tabular}
\caption{Errors and experimental order of convergence for $\lambda- \frac{T_h u_h^\epsilon - Tu }{\epsilon}$ in the numerical experiment, Subsection \ref{sec:PenaltyExperiment}.}
\label{table:LambdaPenalty}
\end{table}

\subsection{Surface numerical and penalty experiment}\label{sec:MixedExperiment}
We now couple $\epsilon$ and $h$, we take $\epsilon \approx Ch^2$.
The same experiment as in Subsections \ref{sec:surfaceNumericalExperiment} and \ref{sec:PenaltyExperiment} is used.
The $E$ and $EOC$ are calculated in terms of the grid size $h$. 
In Tables \ref{table:UMixed} and \ref{table:WMixed} we see that the errors are consistent with the results of of Corollary \ref{cor:ApplicationEpsilonConvergence},  Theorem \ref{thm:SurfaceDiscreteLagrange} and Theorem \ref{thm:surfaceDiscretePenalty}.

\begin{table}[h]
\begin{tabular}{|l|l|l|l|l|l|}
\hline
$h$	&	$\epsilon$ & $E_{L^2}$ & $E_{H^1}$ & $EOC_{L^2}$ & $EOC_{H^1}$ \\ \hline
0.311152	&	0.2	&	0.182674	&	0.193639	&	--	&	--
\\	\hline
0.156914	&	0.05	&	0.0471203	&	0.063719	&	1.97931	&	1.62364
\\	\hline
0.0786276	&	0.0125	&	0.0118776	&	0.0248031	&	1.99435	&	1.36548
\\	\hline
0.0393352	&	0.003125	&	0.00296094	&	0.0112296	&	2.00569	&	1.14411
\\	\hline
0.0196703	&	0.00078125	&	0.000776274	&	0.00546128	&	1.9318	&	1.0402
\\	\hline
\end{tabular}
\caption{Errors and experimental orders of convergence for $u-(u^\epsilon_h)^l$ in the numerical experiment, Subsection \ref{sec:MixedExperiment}.}
\label{table:UMixed}
\end{table}\begin{table}[h]
\begin{tabular}{|l|l|l|l|l|l|}
\hline
$h$	&	$\epsilon$ & $E_{L^2}$ & $E_{H^1}$ & $EOC_{L^2}$ & $EOC_{W^{1,\frac{4}{3}}}$ \\ \hline
0.311152	&	0.2	&	0.337009	&	0.488743	&	--	&	--
\\	\hline
0.156914	&	0.05	&	0.104921	&	0.253073	&	1.70455	&	0.961402
\\	\hline
0.0786276	&	0.0125	&	0.0310131	&	0.140968	&	1.76387	&	0.846845
\\	\hline
0.0393352	&	0.003125	&	0.00893405	&	0.0271549	&	1.79691	&	0.667151
\\	\hline
0.0196703	&	0.00078125	&	0.00300472	&	0.00677873	&	1.57239	&	0.54459
\\	\hline	
\end{tabular}
\caption{Errors and experimental orders of convergence for $w-(w^\epsilon_h)^l$ in the numerical experiment, Subsection \ref{sec:MixedExperiment}.}
\label{table:WMixed}
\end{table}
\newpage


\end{document}